\documentclass[12pt,twoside]{amsart}
\usepackage{latexsym,amsfonts,amsmath,amssymb}
\usepackage{enumerate}
\numberwithin{equation}{section}

\makeatletter
\@namedef{subjclassname@2020}{%
  \textup{2020} Mathematics Subject Classification}
\makeatother

\theoremstyle{plain}
\newtheorem{proposition}{Proposition}[section]
\newtheorem{theorem}[proposition]{Theorem}
\newtheorem{lemma}[proposition]{Lemma}

\newtheorem{definition}[proposition]{Definition}

\newtheorem{remark}[proposition]{Remark}

\catcode`\@=11

\renewcommand{\section}%
   {\setcounter{equation}{0}\@startsection {section}{1}{\z@}{-3.5ex plus -1ex
  minus -.2ex}{2.3ex plus .2ex}{\Large\bf}}

\def\N{\mathbb N}

\newcommand{\RR}{\mathbb{R}}
\newcommand{\CC}{\mathbb{C}}
\newcommand{\NN}{\mathbb{N}}

\newcommand{\supp}{\operatorname{supp}}

\newcommand{\beqsn}{\arraycolsep1.5pt\begin{eqnarray*}}
\newcommand{\eeqsn}{\end{eqnarray*}\arraycolsep5pt}
\newcommand{\beqs}{\arraycolsep1.5pt\begin{eqnarray}}
\newcommand{\eeqs}{\end{eqnarray}\arraycolsep5pt}

\title{On the inclusion relations of global ultradifferentiable classes defined by weight matrices}

\author[Boiti]{Chiara Boiti}
\address{
Dipartimento di Matematica e Informatica \\Universit\`a di Ferrara\\
Via Ma\-chia\-vel\-li n.~30\\
I-44121 Ferrara\\
Italy}
\email{chiara.boiti@unife.it}

\author[Jornet]{David Jornet}
\address{
Instituto Universitario de Matem\'atica Pura y Aplicada IUMPA\\
Universitat Po\-li\-t\`ecni\-ca de Val\`encia\\
Camino de Vera, s/n\\
E-46022 Valencia\\
Spain}
\email{djornet@mat.upv.es}

\author[Oliaro]{Alessandro Oliaro}
\address{Dipartimento di Matematica\\ Universit\`a di Torino\\
 Via Carlo Alberto n.~10\\ I-10123 Torino\\ Italy}
 \email{alessandro.oliaro@unito.it}

\author[Schindl]{Gerhard Schindl}
\address{Fakult\"at f\"ur Mathematik\\ Universit\"at Wien\\
Oskar-Morgenstern-Platz n.~1\\ A-1090 Wien\\ Austria}
 \email{gerhard.schindl@univie.ac.at}

\begin{document}

\keywords{Gelfand-Shilov classes, weight sequences, weight functions, weight matrices, sequence spaces}
\subjclass[2020]{46A04, 46A45, 26E10}

\begin{abstract}
We study and characterize the inclusion relations of global classes in the general weight matrix framework in terms of growth relations for the defining weight matrices. We consider the Roumieu and Beurling cases, and as a particular case we also treat the classical weight function and weight sequence cases. Moreover, we construct a weight sequence which is oscillating around any weight sequence which satisfies some minimal conditions and, in particular, around the critical weight  sequence $(p!)^{1/2}$, related with the non-triviality  of the classes. Finally, we also obtain comparison results both on classes defined by weight functions that can be defined by weight sequences and conversely.
\end{abstract}

\maketitle


\markboth{\sc  On the inclusion relations of global ultradifferentiable classes\ldots}
 {\sc C.~Boiti, D.~Jornet, A.~Oliaro, G.~Schindl}

\section{Introduction}

We deal with global classes of ultradifferentiable functions defined by weight matrices and study and characterize different inclusion relations between these classes. There are basically two ways to introduce the classes of ultradifferentiable functions, the point of view of Komatsu~\cite{Komatsu73}, based in previous ideas of Carleman, which pays attention to the growth of the derivatives on compact sets modulated by a sequence $(M_p)_p$ of positive numbers, or the point of view of Bj\"orck~\cite{Bjorck66}, based in ideas of Beurling~\cite{Beurling61}, who used a weight function to estimate the growth of the Fourier transform of compactly supported functions. Braun, Meise and Taylor~\cite{BraunMeiseTaylor90} unified these points of view by introducing weight functions which allow the use of convex analysis techniques. In terms of their topological structure, the classes are of two types, of Beurling type,  which are classes whose topology looks like the topology of the space of all smooth functions, and of Roumieu type, whose topology looks like that of the space of real-analytic functions.

 More recently, Rainer and Schindl~\cite{compositionpaper} introduced weight matrices to study when the spaces of ultradifferentiable functions are closed under composition treating at the same time the classes in the sense of Komatsu (estimates of the derivatives with a sequence) and in the sense of Braun, Meise and Taylor (estimates of the derivatives via a weight function). They also studied intersections and inclusion relations of the classes in the local sense (i.e. when the estimates are given over the compact sets of a given open set). Since then several papers using weight matrices have been published. We mention, for instance, \cite{nuclearglobal2,almostanalytic,furdos2022the,mixedramisurj}, and the references therein. However, the characterization of inclusion relations in global classes of ultradifferentiable functions, i.e. classes where the estimates on the derivatives are taken in the whole of $\mathbb{R}^d$, has not been investigated yet. In this paper (Section~\ref{characterizationsection}) we characterize the inclusion relations of global classes defined by weight matrices using the isomorphisms introduced in \cite[Section 5]{nuclearglobal2}. Moreover, given a weight sequence we construct an oscillating weight sequence around the first one to have examples of the opposite situation of the inclusion relations. In particular, in Section~\ref{oscillating} we construct an oscillating weight sequence around the sequence $(p!)^{1/2}$, which is very related to the non-triviality of the corresponding ultradifferentiable class  (see Remark~\ref{oscillatingconsequence}). We begin with some notation (Section~\ref{notation}) and continue in Section~\ref{characterizationsection} with the weight function case and  the more general weighted matrix case. In Section~\ref{comparisonofGSclasses} we compare the classes when defined by weight functions and sequences in the spirit of \cite{BonetMeiseMelikhov07}. In Section~\ref{alternativesection} we give alternative assumptions (non-quasianalyticity), and accordingly different techiques of demonstration, for the inclusion relations. And finally, as an appendix we analyze  the extreme case of the weight function $t\mapsto\log t$ which yields the Schwartz class in our setting.


\section{Notation}\label{notation}
\subsection{Weight sequences}
Let $\mathbf{M}=(M_p)_{p\in\NN_0}$ be a sequence of positive real numbers where we have set $\NN_0:=\NN\cup\{0\}$. A sequence $(M_p)_p$ is called {\em normalized} if $M_1\ge M_0=1$ (without loss of generality).
We say that $\mathbf{M}$ satisfies the {\em logarithmic convexity} condition, i.e.  assumption
$(M1)$
of \cite{Komatsu73}, if
\begin{eqnarray}
\label{M1}
M_p^2\leq M_{p-1}M_{p+1},\qquad p\in\NN.
\end{eqnarray}
This is equivalent to the fact that the sequence of quotients $\mu_p:=\frac{M_p}{M_{p-1}}$, $p\in\NN$, is nondecreasing and we set $\mu_0:=1$. If $\mathbf{M}$ is normalized and log-convex, then
\begin{equation}\label{algebra}
\forall\;p,q\in\NN_0:\;\;\;M_pM_q\le M_{p+q};
\end{equation}
see e.g.  \cite[Lemma 2.0.6]{diploma}. Moreover, in this case, $\mathbf{M}$ is nondecreasing because $\mu_p\ge\mu_1\ge 1$ for all $p\in\NN$.

We say that $\mathbf{M}$ satisfies {\em derivation closedness}, i.e. condition $(M2)'$ of \cite{Komatsu73}, if
\begin{eqnarray}
\label{M2prime}
\exists\;D\ge 1\;\;\;M_{p+1}\le D^{p+1}M_p,\qquad p\in\NN_0,
\end{eqnarray}
and $\mathbf{M}$ satisfies the stronger condition of {\em moderate growth}, i.e. condition $(M2)$
of \cite{Komatsu73}, if
\begin{eqnarray}
\label{M2}
\exists\;C\ge 1\;\;\;M_{p+q}\le C^{p+q}M_pM_q,\qquad p,q\in\NN_0.
\end{eqnarray}

For convenience we set
$${\mathcal{LC}}:=\{\mathbf{M}\in\RR_{>0}^{\NN}:\;\mathbf{M}\;\text{is normalized, log-convex},\;\lim_{p\rightarrow\infty}(M_p)^{1/p}=\infty\},$$
where $\RR_{>0}^{\NN}$ denotes the set of strictly positive sequences.

For a normalized sequence
$\mathbf{M}=(M_p)_p$ the {\em associated function} is denoted by
\begin{eqnarray}
\label{assofunc}
\omega_{\mathbf{M}}(t)=\sup_{p\in \NN_0}\log\frac{|t|^p}{M_p},\qquad t\in\RR,
\end{eqnarray}
with the convention that $0^0:=1$ and $\log0:=-\infty$.
Note that
condition $(M_p)^{1/p}\to+\infty$ is equivalent to $\omega_{\mathbf{M}}(t)<+\infty$ by
\cite[Rem.~1]{nuclearglobal2}.

If $\mathbf{M}\in{\mathcal{LC}}$, then we can compute $\mathbf{M}$ by involving $\omega_{\mathbf{M}}$ as follows, see \cite[Chapitre I, 1.4, 1.8]{mandelbrojtbook} (and also \cite[Prop. 3.2]{Komatsu73}):
\begin{equation}\label{Prop32Komatsu}
M_p=\sup_{t\ge 0}\frac{t^p}{\exp(\omega_{\mathbf{M}}(t))},\;\;\;p\in\NN_0.
\end{equation}

Given two (normalized) sequences $\mathbf{M}$ and $\mathbf{N}$ we write $\mathbf{M}\preceq\mathbf{N}$, if
$$\exists\;C\ge 1\;\forall\;p\in\NN_0:\;\;\;M_p\le C^pN_p.$$
If $\mathbf{M}\preceq\mathbf{N}$ and $\mathbf{N}\preceq\mathbf{M}$, then we write $\mathbf{M}\approx\mathbf{N}$ and say that the sequences $\mathbf{M}$ and $\mathbf{N}$ are equivalent. Moreover we write $\mathbf{M}\vartriangleleft\mathbf{N}$ if
$$\forall\;h>0\;\exists\;C_h\ge 1\;\forall\;p\in\NN_0:\;\;\;M_p\le C_hh^pN_p\Longleftrightarrow\lim_{p\rightarrow\infty}\left(\frac{M_p}{N_p}\right)^{1/p}=0.$$
Next we recall \cite[Lemmas 3.8, 3.10]{Komatsu73} transferring these growth relations to the associated functions: Given $\mathbf{M},\mathbf{N}\in{\mathcal{LC}}$ we have
$$\mathbf{M}\preceq\mathbf{N}\Longleftrightarrow\exists\;A,B\ge 1\;\forall\;t\ge 0:\;\,\;\omega_{\mathbf{N}}(t)\le\omega_{\mathbf{M}}(At)+B,$$
and
$$\mathbf{M}\vartriangleleft\mathbf{N}\Longleftrightarrow\forall\;A>0\;\exists\;B\ge 1\;\forall\;t\ge 0:\;\,\;\omega_{\mathbf{N}}(t)\le\omega_{\mathbf{M}}(At)+B.$$
The implications $\Rightarrow$ are clear by definition, for the converse one uses the fact that the sequences are log-convex and \cite[$(3.2)$]{Komatsu73}.\vspace{6pt}

Similar conditions can be considered for sequences $\mathbf{M}=(M_\alpha)_{\alpha\in\NN_0^d}$ of
positive real numbers for multi-indices $\alpha\in\NN_0^d$ (see \cite{nuclearglobal2} for more details). In particular, normalization is given by $M_0=1$, $\mathbf{M}\preceq\mathbf{N}$ is given by  $M_{\alpha}\le C^{|\alpha|}N_{\alpha}$ for some $C\ge 1$ and all $\alpha\in\NN_0^d$ and similarly $\mathbf{M}\vartriangleleft\mathbf{N}$ by $\lim_{|\alpha|\rightarrow\infty}\left(\frac{M_{\alpha}}{N_{\alpha}}\right)^{1/|\alpha|}=0$.
However, the notion of logarithmic convexity is quite delicate in the anisotropic case, and we refer to \cite{BJOS-cvxminorant} for more details.



\subsection{General weight matrices}

Next we recall from \cite[Sect. 3]{nuclearglobal2} the notion of weight matrices and global ultradifferentiable functions in the matrix weighted setting. Let
\begin{equation}
\label{defcalM}
\begin{split}
\qquad \mathcal{M}:=\{(\mathbf{M}^{(\lambda)})_{\lambda>0}:\ & \mathbf{M}^{(\lambda)}=
(M^{(\lambda)}_\alpha)_{\alpha\in\NN_0^d},\
M^{(\lambda)}_0=1, \\
&\mathbf{M}^{(\lambda)}\leq\mathbf{M}^{(\kappa)}\,\mbox{for all}\,0<\lambda\leq\kappa\},
\end{split}
\end{equation}
where $\mathbf{M}^{(\lambda)}\leq\mathbf{M}^{(\kappa)}$ means that
${M}^{(\lambda)}\leq{M}^{(\kappa)}$ for all $\alpha\in\N_0^d$.
We call $\mathcal{M}$ a {\em weight matrix} and we say that it is {\itshape constant} if $\mathbf{M}^{(\lambda)}\approx\mathbf{M}^{(\kappa)}$ for all $\lambda,\kappa>0$.
In the one-dimensional case, we call $\mathcal{M}$ {\itshape standard log-convex} if $\mathbf{M}^{(\lambda)}\in{\mathcal{LC}}$ for any $\lambda>0$.\vspace{6pt}

Given two weight matrices $\mathcal{M}=\{(\mathbf{M}^{(\lambda)})_{\lambda>0}\}$ and $\mathcal{N}=\{(\mathbf{N}^{(\lambda)})_{\lambda>0}\}$ we define the following three relevant growth conditions based on the weight sequence notation.

We write
\beqsn
\mathcal{M}(\preceq)\mathcal{N}&&\qquad\quad\mbox{if}\qquad
\forall\;\lambda>0\;\exists\;\kappa>0:\;\;\;\mathbf{M}^{(\kappa)}\preceq\mathbf{N}^{(\lambda)},\\
\mathcal{M}\{\preceq\}\mathcal{N}&&\qquad\quad\mbox{if}\qquad
\forall\;\lambda>0\;\exists\;\kappa>0:\;\;\;\mathbf{M}^{(\lambda)}\preceq\mathbf{N}^{(\kappa)},\\
\mathcal{M}\vartriangleleft\mathcal{N}&&\qquad\quad\mbox{if}\qquad
\forall\;\lambda>0\;\forall\;\kappa>0:\;\;\;\mathbf{M}^{(\lambda)}\vartriangleleft\mathbf{N}^{(\kappa)}.\eeqsn

Denoting by $\|\cdot\|_\infty$ the supremum norm, for a normalized weight sequence
$\mathbf{M}$ we consider the following spaces of weighted rapidly decreasing
global ultradifferentiable functions
\begin{align*}
\mathcal{S}_{\{\mathbf{M}\}}(\RR^d):=&\Big\{f\in C^{\infty}(\RR^d):\  \exists C,h>0, \ \
\|f\|_{\infty,\mathbf{M},h}:=\sup_{\alpha,\beta\in\NN_0^d}
\frac{\|x^\alpha\partial^\beta f\|_\infty}{h^{|\alpha+\beta|}M_{\alpha+\beta}}\leq C\Big\},\\
\mathcal{S}_{(\mathbf{M})}(\RR^d):=&\{f\in C^\infty(\RR^d):\  \forall h>0\ \exists C_h>0, \ \
\|f\|_{\infty,\mathbf{M},h}\leq C_h\},
\end{align*}
endowed with the inductive limit topology in the Roumieu setting (which may be thought countable if we take $h\in\NN$) and with the projective limit topology in the Beurling setting (countable for $h^{-1}\in\NN$). Next, the matrix type spaces are defined as follows:
\begin{align*}
\mathcal{S}_{\{\mathcal{M}\}}(\RR^d):=&
\bigcup_{\lambda>0}\mathcal{S}_{\{\mathbf{M}^{(\lambda)}\}}
=\{f\in C^\infty(\RR^d):\  \exists C,h,\lambda>0,\ \
\|f\|_{\infty,\mathbf{M}^{(\lambda)},h}\leq C\},\\
\mathcal{S}_{(\mathcal{M})}(\RR^d):=&
\bigcap_{\lambda>0}\mathcal{S}_{(\mathbf{M}^{(\lambda)})}\\
=&\{f\in C^\infty(\RR^d):\  \forall h,\lambda>0\ \exists C_{\lambda,h}>0,\ \
\|f\|_{\infty,\mathbf{M}^{(\lambda)},h}\leq C_{\lambda,h}\},
\end{align*}
again endowed with the inductive limit topology in the Roumieu setting (which may be thought countable if we take $\lambda,h\in\NN$) and endowed with the projective limit topology in the Beurling setting (countable for $\lambda^{-1},h^{-1}\in\NN$).

We denote by $\mathcal{E}_{\{\mathbf{M}\}}, \mathcal{E}_{(\mathbf{M})},
\mathcal{E}_{\{\mathcal{M}\}}, \mathcal{E}_{(\mathcal{M})}$ the analogous (local) classes
of ultradifferentiable functions replacing $\|x^\alpha\partial^\beta f\|_\infty$ by the supremum of $\partial^\beta f$ on compact sets (and then take the projective limit over compact sets). Moreover
 $\mathcal{D}_{\{\mathbf{M}\}}, \mathcal{D}_{(\mathbf{M})},
\mathcal{D}_{\{\mathcal{M}\}}, \mathcal{D}_{(\mathcal{M})}$ denote the corresponding classes of ultradifferentiable functions with compact support.
We refer to \cite{BonetMeiseMelikhov07} and \cite{compositionpaper} for precise definitions of such classes.

We collect here some conditions, already introduced in \cite{nuclearglobal2} and motivated
by the assumptions in \cite{Langenbruch06}. In the Roumieu case we consider
\begin{equation}
\label{12L2R}
\begin{split}
&\forall\lambda>0\ \exists\;\kappa\geq\lambda, B,C,H>0\ \forall\alpha,\beta\in\NN_0^d: \\
&\alpha^{\alpha/2}M^{(\lambda)}_\beta\leq BC^{|\alpha|}H^{|\alpha+\beta|}
M^{(\kappa)}_{\alpha+\beta},
\end{split}
\end{equation}
\begin{eqnarray}
\label{37LR}
&&\forall\;\lambda>0\;\exists\;\kappa\geq\lambda, A\geq1\ \forall\alpha,\beta\in\NN_0^d:\ \ M^{(\lambda)}_{\alpha}
M^{(\lambda)}_{\beta}\leq A^{|\alpha+\beta|}M^{(\kappa)}_{\alpha+\beta},
\end{eqnarray}
\begin{eqnarray}
\label{M2'R}
\forall\lambda>0\;\exists\kappa\geq\lambda,A\geq1\forall\alpha\in\NN_0^d,1\leq j\leq d:\
M^{(\lambda)}_{\alpha+e_j}\leq A^{|\alpha|+1}M^{(\kappa)}_\alpha,
\end{eqnarray}
\begin{eqnarray}
\label{M2R}
\forall\lambda>0\exists\kappa\geq\lambda,A\geq1\forall\alpha,\beta\in\NN_0^d:\
M^{(\lambda)}_{\alpha+\beta}\leq A^{|\alpha+\beta|}
M^{(\kappa)}_{\alpha} M^{(\kappa)}_{\beta},
\end{eqnarray}
and in the Beurling case
\begin{equation}
\label{12L2B}
\begin{split}
&\forall\;\lambda>0\ \exists\;0<\kappa\leq\lambda, H>0\ \forall C>0\ \exists B>0\  \forall\alpha,\beta\in\NN_0^d:\\
&\alpha^{\alpha/2}M^{(\kappa)}_\beta\leq BC^{|\alpha|}H^{|\alpha+\beta|}
M^{(\lambda)}_{\alpha+\beta},
\end{split}
\end{equation}
\begin{eqnarray}
\label{37LB}
\forall\;\lambda>0\;\exists\;0<\kappa\leq\lambda, A\geq1\ \forall\alpha,\beta\in\NN_0^d:\ \  M^{(\kappa)}_{\alpha}
M^{(\kappa)}_{\beta}\leq A^{|\alpha+\beta|}M^{(\lambda)}_{\alpha+\beta},
\end{eqnarray}
\begin{eqnarray}
\label{M2'B}
&&\forall\lambda>0\;\exists0<\kappa\leq\lambda,A\geq1\forall\alpha\in\NN_0^d,1\leq j\leq d:\
M^{(\kappa)}_{\alpha+e_j}\leq A^{|\alpha|+1}M^{(\lambda)}_\alpha,
\end{eqnarray}
\begin{eqnarray}
\label{M2B}
&&\forall\lambda>0\exists0<\kappa\leq\lambda,A\geq1\forall\alpha,\beta\in\NN_0^d:\
M^{(\kappa)}_{\alpha+\beta}\leq A^{|\alpha+\beta|}M^{(\lambda)}_{\alpha}
M^{(\lambda)}_{\beta}.
\end{eqnarray}

We summarize now some consequences for a given weight matrix $\mathcal{M}$ as defined in \eqref{defcalM}:

\begin{itemize}
\item[$(i)$] By \cite[Proposition 1]{nuclearglobal2} the spaces $\mathcal{S}_{\{\mathcal{M}\}}(\RR^d)$ resp. $\mathcal{S}_{(\mathcal{M})}(\RR^d)$ can be equivalently defined in terms of the system of (weighted) $L^2$-seminorms when assuming \eqref{12L2R} and \eqref{M2'R} in the Roumieu case, resp. \eqref{12L2B} and \eqref{M2'B} in the Beurling case.

\item[$(ii)$] If $\mathcal{M}$ satisfies \eqref{37LR} and \eqref{M2R}, then we can replace in the definition of $\mathcal{S}_{\{\mathcal{M}\}}(\RR^d)$ the seminorm $\|\cdot\|_{\infty,\mathbf{M}^{(\lambda)},h}$ by
\begin{equation}\label{separatedgrowth}
\sup_{\alpha,\beta\in\NN_0^d}
\frac{\|x^\alpha\partial^\beta f\|_\infty}{h^{|\alpha+\beta|}M^{(\lambda)}_{\alpha}M^{(\lambda)}_{\beta}}.
\end{equation}
We have an analogous statement for the class $\mathcal{S}_{(\mathcal{M})}(\RR^d)$ when assuming \eqref{37LB} and \eqref{M2B}. When we define the spaces $\mathcal{S}_{\{\mathcal{M}\}}(\RR^d)$ or $\mathcal{S}_{(\mathcal{M})}(\RR^d)$ with the weighted $L^2$ norms, the similar property holds.
\end{itemize}

From now on we make use of the following conventions: when the spaces are defined in terms of $\|\cdot\|_{\infty,\mathbf{M}^{(\lambda)},h}$ we will always write ``joint growth at infinity''; when they are defined in terms of the seminorms from \eqref{separatedgrowth}, then we will write ``separated growth at infinity''; frequently we omit writing the set $\RR^d$ when the context is clear.
%


\subsection{Weight functions}
\label{weightfunctions}

 \begin{definition}
\label{defomega}
A {\em weight function} is a continuous increasing function
$\omega\!:[0,+\infty)\rightarrow[0,+\infty)$ such that
\begin{itemize}
\item[$(\alpha)$]
$\exists L\geq1\ \forall t\geq 0:\ \omega(2t)\leq L(\omega(t)+1)$;
\item[$(\beta)$]
$\omega(t)=O(t^2)$ as $t\rightarrow+\infty$;
\item[$(\gamma)$]
$\log t=o(\omega(t))$ as $t\rightarrow+\infty$;
\item[$(\delta)$]
$\varphi_\omega(t):=\omega(e^t)$ is convex on $[0,+\infty)$.
\end{itemize}
Then we define $\omega(t):=\omega(|t|)$ if $t\in\RR^d$.

We call $\omega$ a {\em general weight function,} if $\omega$ satisfies all listed properties except $(\beta)$.
\end{definition}

\vspace{6pt}

It is not restrictive to assume $\omega|_{[0,1]}\equiv0$ (normalization). As usual, we define the Young conjugate $\varphi^*_{\omega}$ of $\varphi_{\omega}$ by
\begin{eqnarray*}
\varphi^*_{\omega}(s):=\sup_{t\geq0}\{ts-\varphi_{\omega}(t)\},\;\;\;s\ge 0,
\end{eqnarray*}
which is an increasing convex function such that $\varphi^{**}_{\omega}=\varphi_{\omega}$ and $s\mapsto\frac{\varphi^{*}_{\omega}(s)}{s}$ is increasing. Condition $(\gamma)$ guarantees that $\varphi^*_{\omega}$ is finite (see Appendix \ref{logweight}).


We introduce the following growth relations between two (general) weight functions arising naturally in the ultradifferentiable framework: We write
$$\omega\preceq\sigma\quad\mbox{if}\quad\sigma(t)=O(\omega(t)),\;\;\;t\rightarrow\infty,$$
 and
$$\omega\vartriangleleft\sigma\quad\mbox{if}\quad\sigma(t)=o(\omega(t)),\;\;\;t\rightarrow\infty.$$
If $\omega\preceq\sigma$ and $\sigma\preceq\omega$ are valid, then we write $\omega\sim\sigma$ and call the weights {\itshape equivalent.}

For any given (general) weight function we set
\begin{eqnarray}
\label{Wlambda}
W^{(\lambda)}_\alpha:=e^{\frac{1}{\lambda}\varphi^*_{\omega}(\lambda |\alpha|)},
\qquad\lambda>0,\alpha\in\NN_0^d,
\end{eqnarray}
and consider the weight matrix
\begin{eqnarray}
\label{calMomega}
\mathcal{M}_{\omega}:=(\mathbf{W}^{(\lambda)})_{\lambda>0}=(W^{(\lambda)}_\alpha)_{\lambda>0,\,\alpha\in\NN_0^d}.
\end{eqnarray}

We recall now the following result, which was proved in \cite[Lemma 11]{nuclearglobal2} for a weight function, but which can be stated for a {\em general weight function}, since assumption $(\beta)$
was not needed in the proof.


\begin{lemma}\label{Lemma61}
Let $\omega$ be a (general) weight function. Then
$\mathcal{M}_{\omega}$ satisfies the following properties:
\begin{enumerate}
\item[(i)]
$W^{(\lambda)}_0=1,\quad\lambda>0$;
\item[(ii)]
$(W^{(\lambda)}_\alpha)^2\leq W^{(\lambda)}_{\alpha-e_i}W^{(\lambda)}_{\alpha+e_i},\quad
\lambda>0,\alpha\in\NN^d_0$ with $\alpha_i\neq 0$, and $i=1,\dots,d$;
\item[(iii)]
$\mathbf{W}^{(\kappa)}\leq\mathbf{W}^{(\lambda)},\quad 0<\kappa\leq\lambda$;
\item[(iv)]
$W^{(\lambda)}_{\alpha+\beta}\leq W^{(2\lambda)}_\alpha W^{(2\lambda)}_\beta,\quad\lambda>0,\alpha,\beta\in\NN_0^d$;
\item[(v)]
$\forall h>0\ \exists A\geq1\ \forall\lambda>0\ \exists D\geq1\ \forall \alpha\in\NN_0^d:\ \
\ h^{|\alpha|}W^{(\lambda)}_\alpha\leq DW^{(A\lambda)}_\alpha;$
\item[(vi)] Both conditions \eqref{M2'R} and \eqref{M2'B} are valid.
\item[(vii)]
Conditions \eqref{37LR} and \eqref{37LB} are satisfied for $\kappa=\lambda$ and
$A=1$.
\end{enumerate}
\end{lemma}

The spaces of rapidly decreasing $\omega$-ultradifferentiable
functions are then defined as follows: In the Roumieu case
\begin{align*}
\mathcal{S}_{\{\omega\}}(\RR^d)&:=\big\{f\in C^{\infty}(\RR^d):\ \exists\lambda>0,C>0\ \mbox{s.t.}\sup_{\alpha,\beta\in\NN_0^d}\|x^{\alpha}\partial^{\beta} f\|_{\infty}
e^{-\frac{1}{\lambda}\varphi^*_\omega(\lambda |\alpha+\beta|)}\leq C\big\}
\\&
=\big\{f\in C^\infty(\RR^d):\ \exists\lambda>0,C>0\ \mbox{s.t.}\|f\|_{\infty,\mathbf{W}^{(\lambda)}}:=
\sup_{\alpha,\beta\in\NN_0^d}\frac{\|x^\alpha\partial^\beta f\|_\infty}{W^{(\lambda)}_{\alpha+\beta}}\leq C\big\},
\end{align*}
and in the Beurling case
$$\mathcal{S}_{(\omega)}(\RR^d):=\big\{f\in C^{\infty}(\RR^d):\ \forall\lambda>0\,\exists C_{\lambda}>0:\
\|f\|_{\infty,\mathbf{W}^{(\lambda)}}\leq C_\lambda\big\}.$$

From $(iv)$, $(vii)$ in Lemma \ref{Lemma61} we have that equivalently the classes can be treated by separated growth at infinity, i.e.
\begin{eqnarray*}
\mathcal{S}_{\{\omega\}}(\RR^d)=\big\{f\in C^{\infty}(\RR^d):\ \exists\lambda>0,C>0:\
\sup_{\alpha,\beta\in\NN_0^d}
\frac{\|x^{\alpha}\partial^{\beta}f\|_\infty}{W^{(\lambda)}_{\alpha}
W^{(\lambda)}_{\beta}}\leq C\big\}
\end{eqnarray*}
and
\begin{eqnarray*}
\mathcal{S}_{(\omega)}(\RR^d)=\big\{f\in C^{\infty}(\RR^d):\ \forall\lambda>0\,\exists C_\lambda>0:\
\sup_{\alpha,\beta\in\NN_0^d}\frac{\|x^{\alpha}\partial^{\beta}f\|_\infty}{W^{(\lambda)}_{\alpha}
W^{(\lambda)}_{\beta}}\leq C_\lambda\big\}.
\end{eqnarray*}

We can also insert $h^{|\alpha+\beta|}$ at the denominator (for some $h>0$ in the
Roumieu case and for all $h>0$ in the Beurling case) by $(v)$ in Lemma \ref{Lemma61}. In particular, we finally recall from \cite[Proposition 5]{nuclearglobal2} (where again assumption $(\beta)$ was not necessary) that, analogously as in the ultradifferentiable setting, we can use the associated weight function in order to have an alternative and useful description of the classes defined by weight matrices:

\begin{proposition}\label{Proposition61}
Let $\omega$ be a (general) weight function and $\mathcal{M}_{\omega}$ be the weight matrix defined in  \eqref{Wlambda}, \eqref{calMomega}. We have
$$\mathcal{S}_{\{\mathcal{M}_{\omega}\}}(\RR^d)=\mathcal{S}_{\{\omega\}}(\RR^d)\hspace{20pt}\text{and}\hspace{20pt} \mathcal{S}_{(\mathcal{M}_{\omega})}(\RR^d)=\mathcal{S}_{(\omega)}(\RR^d),$$
and both equalities are also topological.
\end{proposition}


We refer to \cite{BJO-JFA} for a more complete characterization of such spaces, and to
\cite{BraunMeiseTaylor90} for the analogous spaces $\mathcal{E}_{\{\omega\}}/\mathcal{D}_{\{\omega\}}$ and $\mathcal{E}_{(\omega)}/\mathcal{D}_{(\omega)}$ of ultradifferentiable functions/with compact support.


\section{Oscillating sequences and a critical example case}\label{oscillating}
The aim of this section is to construct explicitly a weight sequence $\mathbf{M}$ which oscillates around a given fixed sequence $\mathbf{N}\in{\mathcal{LC}}$. We assume for $\mathbf{N}$ some more basic growth properties and show that these requirements can be transferred to $\mathbf{M}$, too. Moreover, these conditions yield the fact that the function $\omega\equiv\omega_{\mathbf{M}}$ is also oscillating around the weight $\omega_{\mathbf{N}}$. Since by construction $\mathbf{M}\in{\mathcal{LC}}$ we focus on the one-dimensional situation (or, equivalently, on the isotropic case, i.e. $M_{\alpha}:=M_{|\alpha|}$). As a special case we apply this to $\mathbf{N}\equiv\mathbf{G}^{1/2}:=(p!^{1/2})_{p\in\NN}$ and the corresponding weight function $\omega(t)=t^2$. This is a crucial case since it is related to the problem of
non-triviality of $\mathcal{S}_{(\omega)}$ and $\mathcal{S}_{\{\omega\}}$ (see Remark~\ref{oscillatingconsequence}).

\subsection{Construction of the sequence $\mathbf{M}$}\label{construction}

{\itshape Goal:} We construct the sequence $\mathbf{M}$ in terms of the quotients $\mu_p:=\frac{M_p}{M_{p-1}}$ by putting $M_p:=\prod_{i=1}^p\mu_i$ (and $M_0:=1$, empty product) and consider the associated weight function $\omega_{\mathbf{M}}$. More precisely, the aim is to show that
\begin{itemize}
\item[$(i)$] $\omega_{\mathbf{M}}$ satisfies $(\alpha), (\gamma), (\delta)$,

\item[$(ii)$] $\inf_{p\ge 1}\left(\frac{M_p}{N_p}\right)^{1/p}=0$ and

\item[$(iii)$] $\sup_{p\ge 1}\left(\frac{M_p}{N_p}\right)^{1/p}=+\infty$.
\end{itemize}
Concerning $(i)$, first we recall that for any given $\mathbf{M}\in{\mathcal{LC}}$ the function $\omega_{\mathbf{M}}$ satisfies automatically the basic assumption in Definition \ref{defomega} and conditions $(\gamma)$ and $(\delta)$, see \cite[Chapitre I]{mandelbrojtbook}, \cite[Definition 3.1]{Komatsu73} and \cite[Lemma 12 $(4)\Rightarrow(5)$]{BonetMeiseMelikhov07}.\vspace{6pt}

We start the construction as follows. Let from now on $(\alpha_j)_{j\ge 0}$ be a sequence of positive real numbers such that
\begin{equation}\label{alphagrowthres}
1<\alpha_{\min}:=\inf_{j\ge 0}\alpha_j\le\sup_{j\ge 0}\alpha_j=:\alpha_{\max}<+\infty.
\end{equation}
Moreover let $Q\in\NN$, $Q\ge 2$, be given, arbitrary but fixed. We introduce a new sequence $(\beta_j)_{j\ge 1}$ by
\begin{equation}\label{betadef}
\beta_1=\dots=\beta_{Q-1}:=\alpha_0^{\frac{1}{Q-1}}
\end{equation}
and
\begin{equation}\label{betadef1}
\beta_{Q^n}=\dots=\beta_{Q^{n+1}-1}:=\alpha_n^{\frac{1}{Q^n(Q-1)}},\;\;\;n\in\NN.
\end{equation}
Finally $\mathbf{M}$ is defined via the quotients $(\mu_j)_{j\ge 1}$ as follows: We put
\begin{equation}\label{betadef2}
\mu_1:=c\ge 1,\hspace{30pt}\mu_{j+1}:=\beta_j\mu_j,\;\;\;j\ge 1.
\end{equation}
Using this we have $\frac{\mu_{Qj}}{\mu_j}=\frac{\mu_{j+1}}{\mu_j}\cdots\frac{\mu_{Qj}}{\mu_{Qj-1}}=\beta_j\cdots\beta_{Qj-1}$ for all $j\ge 1$ which yields a product consisting of $j(Q-1)$-many factors.\vspace{6pt}

{\itshape Claim I:} $\alpha_{\min}\le\frac{\mu_{Qj}}{\mu_j}\le\alpha_{\max}$ is valid for all $j\in\NN$.\vspace{6pt}

If $j=1$, then $\frac{\mu_{Qj}}{\mu_j}=\frac{\mu_Q}{\mu_1}=\beta_1\cdots\beta_{Q-1}=\alpha_0^{\frac{Q-1}{Q-1}}=\alpha_0$.

If $j=Q^n$, $n\in\NN$ arbitrary, then $Qj=Q^{n+1}$ and so $\frac{\mu_{Qj}}{\mu_j}=\beta_{Q^n}\cdots\beta_{Q^{n+1}-1}=\alpha_n^{\frac{Q^{n+1}-Q^n}{Q^n(Q-1)}}=\alpha_n$.

If $Q^n<j\le Q^{n+1}-1$, $n\in\NN$ arbitrary, then $Q^{n+1}<Qj\le Q^{n+2}-Q<Q^{n+2}-1$ and we get, for $i=Q^{n+1}-j$,
\begin{align*}
\frac{\mu_{Qj}}{\mu_j}&=\beta_j\cdots\beta_{Qj-1}=\alpha_n^{\frac{i}{Q^n(Q-1)}}\cdot\alpha_{n+1}^{\frac{j(Q-1)-i}{Q^{n+1}(Q-1)}}\ge\alpha_{\min}^{\frac{Qi+Qj-j-i}{Q^{n+1}(Q-1)}}
\\&
=\alpha_{\min}^{\frac{(i+j)(Q-1)}{Q^{n+1}(Q-1)}}=\alpha_{\min}^{\frac{(i+j)}{Q^{n+1}}}=\alpha_{\min}.
\end{align*}
Finally, if $Q^0=1<j\le Q-1$ (when $Q\ge 3$), then we can estimate as before replacing
$\alpha_n$ and $\alpha_{n+1}$ by $\alpha_0$ and $\alpha_1$ respectively. The estimate from above in the claim is obtained analogously when taking $\alpha_{\max}$ instead of $\alpha_{\min}$.\vspace{6pt}

{\itshape Claim II:} $\mathbf{M}\in{\mathcal{LC}}$ holds.

Since by definition and assumption $\alpha_j>1\Leftrightarrow\beta_j>1\Leftrightarrow\mu_{j+1}>\mu_j$ for all $j\in\NN$ we have that $\mathbf{M}$ is log-convex. The previous Claim I yields $\mu_{Qj}\ge\alpha_{\min}\mu_j$ for all $j\in\NN$, thus by iteration $\mu_{Q^nj}\ge\alpha_{\min}^n\mu_j$ for all $n\in\NN$. Consequently we get $\mu_{Q^n}\ge\alpha_{\min}^n\mu_1=\alpha_{\min}^nc\ge\alpha_{\min}^n$ which tends to infinity as $n\rightarrow\infty$ because $\alpha_{\min}>1$ by assumption. This proves $\lim_{j\rightarrow\infty}\mu_j=+\infty$, hence $\lim_{j\rightarrow\infty}(M_j)^{1/j}=+\infty$ follows (e.g. see \cite[p. 104]{compositionpaper}).\vspace{6pt}

Now we start with the definition of $\mathbf{M}$ in terms of the aforementioned construction by using the auxiliary sequences $(\beta_j)_j$ resp. $(\alpha_j)_j$. We put $$\mu_0=\mu_1:=1(=c),$$
so $M_1=M_0=1$ follows which ensures normalization. The idea is now to define $\mathbf{M}$ (via $(\alpha_j)_j$) piece-wise by considering an increasing sequence (of integers) $(k_j)_{j\ge 1}$.
Given a sequence $\mathbf{N}\in{\mathcal{LC}}$ we consider the sequence of its quotients
$$\nu_k=\frac{N_k}{N_{k-1}}, \quad k=1,2,\ldots,$$
and, moreover, we assume that $\mathbf{N}$ satisfies
\begin{equation}\label{cond-fuzzy}
\exists\;Q\in\mathbb{N}:\;\;\;1<\liminf_j\frac{\nu_{Qj}}{\nu_j}\ \ \text{and }\ \ \sup_j\frac{\nu_{2j}}{\nu_j}<+\infty.	
\end{equation}
It is immediate that $Q\ge 2$ in the above requirement and \eqref{cond-fuzzy} is crucial to ensure that $\omega_{\mathbf{M}}$ satisfies $(\alpha)$ (see $(III)$ below) and that $\mathbf{M}$ has moderate growth (see $(IV)$ below).

Let now $Q$ be the parameter according to \eqref{cond-fuzzy} and without loss of generality $Q\ge 3$.
First we set $k_1:=Q$ and
$$\alpha_0=\alpha_1:=4\sqrt{\nu_{k_1}}.$$
Then we have
$$\mu_{k_1}=\mu_Q:=\alpha_0\mu_1=\alpha_0=4\sqrt{\nu_{k_1}}(\ge 4),$$
and for $1<i<k_1$ we have put $\mu_i:=\beta_{i-1}\mu_{i-1}=\alpha_0^{\frac{1}{Q-1}}\mu_{i-1}$ (see \eqref{betadef}).\vspace{6pt}

In the next step we select a number $n_1\in\NN$, $n_1\ge 2$, and put $k_2:=Q^{n_1}k_1=Q^{n_1+1}>k_1$. Here we choose $n_1$ sufficiently large in order to ensure $\nu_{Q^{n_1}k_1}>64\nu_{k_1}$ (note that $\lim_{j\rightarrow\infty}\nu_j=\infty$). Then we set
$$\alpha_2=\dots=\alpha_{n_1}:=64^{-\frac{1}{n_1-1}}\big(\frac{\nu_{k_2}}{\nu_{k_1}}\big)^{\frac{1}{n_1-1}}>1,$$
and get $\alpha_0\cdot\alpha_1\cdot\alpha_2\cdots\alpha_{n_1}=\displaystyle 16\nu_{k_1}\frac{1}{64}\frac{\nu_{k_2}}{\nu_{k_1}}=\frac{1}{4}\nu_{k_2}$. Hence, by  \eqref{betadef1} and \eqref{betadef2} we get $$\frac{\mu_{k_2}}{\mu_{k_1}}=\frac{\mu_{Q^{n_1+1}}}{\mu_Q}=\prod^{Q^{n_1+1}-1}_{i=Q}\frac{\mu_{i+1}}{\mu_i}=\prod_{l=1}^{n_1}\prod_{i=Q^l}^{Q^{l+1}-1}\beta_i=\alpha_1\cdots\alpha_{n_1},$$ and so one has
$$\mu_{k_2}=\mu_{Q^{n_1}k_1}=\mu_{k_1}\cdot\alpha_1\cdots\alpha_{n_1}=\alpha_0\cdot\alpha_1\cdot\alpha_2\cdots\alpha_{n_1}=\frac{1}{4}\nu_{k_2}.$$
Note that for $k_1<i<k_2$ we put $\mu_i:=\alpha_{l+1}^{\frac{1}{Q^{l+1}(Q-1)}}\mu_{i-1}$ whenever $Q^lk_1<i\le Q^{l+1}k_1$, $0\le l\le n_1-1$ (see \eqref{betadef1} and \eqref{betadef2}).\vspace{6pt}

Then select a number $n_2\in\NN$, $n_2\ge 2$, put $k_3:=Q^{n_2}k_2=Q^{n_2+n_1+1}>k_2$ and
$$\alpha_{n_1+1}=\dots=\alpha_{n_1+n_2}:=32^{\frac{1}{n_2}}\big(\frac{\nu_{k_3}}{\nu_{k_2}}\big)^{\frac{1}{n_2}}(\ge 32).$$
Hence we get
$$\alpha_0\cdot\alpha_1\cdot\alpha_2\cdots\alpha_{n_1+n_2}=\frac{1}{4}\nu_{k_2}32\frac{\nu_{k_3}}{\nu_{k_2}}=8\nu_{k_3}.$$ So
$$\mu_{k_3}=\mu_{Q^{n_2}k_2}=\alpha_0\cdot\alpha_1\cdot\alpha_2\cdots\alpha_{n_1+n_2}=8\nu_{k_3},$$
since $\frac{\mu_{k_3}}{\mu_{k_2}}=\prod^{Q^{n_2+n_1+1}-1}_{i=Q^{n_1+1}}\frac{\mu_{i+1}}{\mu_i}=\alpha_{n_1+1}\cdots\alpha_{n_1+n_2}$.

For $k_2<i<k_3$, again according to \eqref{betadef1} and \eqref{betadef2}, we have put $\mu_i:=\alpha_{l+n_1+1}^{\frac{1}{Q^{l+n_1+1}(Q-1)}}\mu_{i-1}$ whenever $Q^lk_2<i\le Q^{l+1}k_2$, $0\le l\le n_2-1$.\vspace{6pt}

And then we proceed as follows:\vspace{6pt}

{\itshape Case I - from odd to even numbers.} Given any $k_j$ with $j\ge 3$ odd, then we select $n_j\in\NN$, $n_j\ge 2$, put $k_{j+1}:=Q^{n_j}k_j=Q^{n_j+\dots+n_1+1}$ and define
$$\alpha_{n_{j-1}+\dots+ n_1+1}=\dots=\alpha_{n_j+n_{j-1}+\dots+n_1}:=(2^j(j+1))^{-\frac{1}{n_j}}\big(\frac{\nu_{k_{j+1}}}{\nu_{k_{j}}}\big)^{\frac{1}{n_j}}.$$
{\itshape Case II - from even to odd numbers.} Given any $k_j$ with $j\ge 4$ even, then we select $n_j\in\NN$, $n_j\ge 2$, put $k_{j+1}:=Q^{n_j}k_j=Q^{n_j+\dots+n_1+1}$ and define
$$\alpha_{n_{j-1}+\dots+ n_1+1}=\dots=\alpha_{n_j+n_{j-1}+\dots+n_1}:=(2^{j+1}j)^{\frac{1}{n_j}}\big(\frac{\nu_{k_{j+1}}}{\nu_{k_{j}}}\big)^{\frac{1}{n_j}}.$$
With these choices, first for all odd $j\ge 3$ (starting with the case $j=3$ from above) one has
\begin{align*}
\alpha_0\alpha_1\alpha_2\cdots\alpha_{n_j+n_{j-1}+\dots+n_1}&=\left(\alpha_0\alpha_1\alpha_2\cdots\alpha_{n_{j-1}+\dots+ n_1}\right)\cdot\alpha_{n_{j-1}+\dots+ n_1+1}\cdots\alpha_{n_j+\dots+n_1}
\\&
=2^{j}\nu_{k_j}\frac{1}{2^j(j+1)}\cdot\frac{\nu_{k_{j+1}}}{\nu_{k_{j}}}=\frac{1}{j+1}\nu_{k_{j+1}},
\end{align*}
and so
\begin{align*}
&\mu_{k_{j+1}}=\mu_{Q^{n_j}k_j}=\alpha_0\alpha_1\alpha_2\cdots\alpha_{n_j+\dots+n_1}=\frac{1}{j+1}\nu_{k_{j+1}}.
\end{align*}
On the other hand, for all even $j\ge 4$ we see
\begin{align*}
\alpha_0\alpha_1\alpha_2\cdots\alpha_{n_j+n_{j-1}+\dots+n_1}&=\left(\alpha_0\alpha_1\alpha_2\cdots\alpha_{n_{j-1}+\dots+ n_1}\right)\cdot\alpha_{n_{j-1}+\dots+ n_1+1}\cdots\alpha_{n_j+\dots+n_1}
\\&
=\frac{1}{j}\nu_{k_j}\cdot 2^{j+1}j \frac{\nu_{k_{j+1}}}{\nu_{k_{j}}}=2^{j+1}\nu_{k_{j+1}},
\end{align*}
and so
\begin{align*}
&\mu_{k_{j+1}}=\mu_{Q^{n_j}k_j}=\alpha_0\alpha_1\alpha_2\cdots\alpha_{n_j+\dots+n_1}=2^{j+1}\nu_{k_{j+1}}.
\end{align*}

Moreover, recall that $\frac{\mu_{k_{j+1}}}{\mu_{k_j}}=\prod_{i=Q^{n_{j-1}+\dots+n_1+1}}^{Q^{n_j+\dots+n_1+1}-1}\frac{\mu_{i+1}}{\mu_i}=\alpha_{n_{j-1}+\dots+n_1+1}\cdots\alpha_{n_j+\dots+n_1}$ and for all $k_j<i<k_{j+1}$, according to \eqref{betadef1} and \eqref{betadef2}, we have set
$$\mu_i:=(\alpha_{l+n_{j-1}+\dots+n_1+1})^{\frac{1}{Q^{l+n_{j-1}+\dots+n_1+1}(Q-1)}}\mu_{i-1},\quad Q^lk_j<i\le Q^{l+1}k_j,\;\;\;0\le l\le n_j-1.$$
\vspace{6pt}

{\itshape Claim III:} \eqref{cond-fuzzy} implies \eqref{alphagrowthres}. First, we treat the upper estimates and note that for Case~I we have $(2^j(j+1))^{-\frac{1}{n_j}}\big(\frac{\nu_{k_{j+1}}}{\nu_{k_{j}}}\big)^{\frac{1}{n_j}}\le\big(\frac{\nu_{k_{j+1}}}{\nu_{k_{j}}}\big)^{\frac{1}{n_j}}\le A$ for some $A\ge 1$. The second estimate is equivalent to $\frac{\nu_{Q^{n_j}k_j}}{\nu_{k_{j}}}\le A^{n_j}$ (recall: $k_{j+1}=Q^{n_j}k_j$) and this is valid because by the second part of \eqref{cond-fuzzy} we have $\frac{\nu_{2j}}{\nu_j}\le B$ for some $B\ge 1$ and all $j\in\NN_0$ and so, iterating this estimate $cn_j$-times with $c\in\NN$ such that $Q\le 2^c$, we have $\frac{\nu_{Q^{n_j}k_j}}{\nu_{k_j}}\le\frac{\nu_{2^{cn_j}k_j}}{\nu_{k_j}}\le B^{cn_j}=A^{n_j}$ with $A:=B^c$. Here the first estimate holds by the log-convexity of $\mathbf{N}$.

For Case II with the expression $(2^{j+1}j)^{\frac{1}{n_j}}\big(\frac{\nu_{k_{j+1}}}{\nu_{k_{j}}}\big)^{\frac{1}{n_j}}$ we get the bound $(2^{j+1}j)^{\frac{1}{n_j}}A$ by the previous comments. And this can be bounded uniformly for all even $j\ge 4$ by some $A_1>A$ when choosing $n_j$ large enough. Therefore note that $A$ is not depending on the choice of $n_j$; it only depends on given (fixed) constants $Q$ and $B$, both depending only on $\mathbf{N}$ via \eqref{cond-fuzzy}. Summarizing, the upper estimate in \eqref{alphagrowthres} is verified for all $j\in\mathbb{N}$ since the remaining cases are only finitely many indices.\vspace{6pt}

Now we treat the lower estimate. By the first part in \eqref{cond-fuzzy} we have that there exists some $\epsilon>0$ such that (by iteration) $\frac{\nu_{k_{j+1}}}{\nu_{k_j}}=\frac{\nu_{Q^{n_j}k_j}}{\nu_{k_j}}\ge(1+\epsilon)^{n_j}$ provided that $k_j$ is chosen sufficiently large. We assume now that we have chosen $k_3$ large enough (for a fixed $\epsilon>0$) and so the above estimate holds for all $j\ge 3$. Now, concerning Case I for all odd $j\ge 3$ we estimate by $(2^j(j+1))^{-\frac{1}{n_j}}\big(\frac{\nu_{k_{j+1}}}{\nu_{k_{j}}}\big)^{\frac{1}{n_j}}\ge(2^j(j+1))^{-\frac{1}{n_j}}(1+\epsilon)>1$ and the last estimate is equivalent to requiring $(1+\epsilon)^{n_j}>2^j(j+1)$. This can be achieved by choosing $n_j$, $j\ge 3$ odd, sufficiently large.

Concerning Case II we observe $(2^{j+1}j)^{\frac{1}{n_j}}\big(\frac{\nu_{k_{j+1}}}{\nu_{k_{j}}}\big)^{\frac{1}{n_j}}\ge\big(\frac{\nu_{k_{j+1}}}{\nu_{k_{j}}}\big)^{\frac{1}{n_j}}\ge(1+\epsilon)$ for all even $j\ge 4$.

To guarantee $\alpha_j>1$ for all $j\in\NN$, i.e. also for $1\le j\le n_1+n_2$, we recall our choice for $n_1$ above.

\vskip\baselineskip
Summarizing, we get:
\begin{itemize}
\item[$(I)$] $\mathbf{M}\in{\mathcal{LC}}$: Normalization is obtained as seen above, log-convexity holds by the fact that $\alpha_j>1$ for all $j\in\NN$ and so the sequence $j\mapsto\mu_j$ is (strictly) increasing. Since for all odd $j\ge 3$ by construction we get $\mu_{k_j}\ge 2^{j+1}\nu_{k_j}$, we see that $\lim_{j\rightarrow\infty}\mu_j=+\infty$ (since $\nu_{k_j}$ is nondecreasing by the logarithmic convexity), and so also $\lim_{j\rightarrow\infty}(M_j)^{1/j}=+\infty$, e.g. see \cite[p. 104]{compositionpaper}.

\item[$(II)$] Moreover, by Claim III we see that $1<\alpha_{\min}\le\alpha_{\max}<+\infty$ when choosing $n_j$ large enough. Thus, by construction and Claim I we have $1<\liminf_{j\rightarrow\infty}\frac{\mu_{Qj}}{\mu_j}\le\limsup_{j\rightarrow\infty}\frac{\mu_{Qj}}{\mu_j}<+\infty$.

\item[$(III)$] The proof of \cite[Lemma 12, $(2)\Rightarrow(4)$]{BonetMeiseMelikhov07} shows that the lower estimate $1<\liminf_{j\rightarrow\infty}\frac{\mu_{Qj}}{\mu_j}$ implies $(\alpha)$ for $\omega_{\mathbf{M}}$. Thus $\omega_{\mathbf{M}}$ has all standard requirements to be a weight function except $(\beta)$, i.e. $\omega_{\mathbf{M}}$ is a general weight function.

\item[$(IV)$] By the upper estimate, $\mathbf{M}$ satisfies \eqref{M2}, see e.g. \cite[Lemma 2.2]{whitneyextensionweightmatrix}). Equivalently, by taking into account \cite[Prop. 3.6]{Komatsu73}, the associated weight function $\omega_{\mathbf{M}}$ satisfies
the following condition
    \begin{equation}\label{omega6}
\exists\;H\ge 1\;\forall\;t\ge 0:\;\;\;2\omega(t)\le\omega(Ht)+H,
\end{equation}
introduced in \cite[Corollary 16 $(3)$]{BonetMeiseMelikhov07} in order to compare
ultradifferentiable spaces defined by weight sequences $(M_p)_p$ and weight functions
$\omega_{\mathbf{M}}$.

\item[$(V)$] Let now $\mathcal{M}_{\omega_{\mathbf{M}}}=\{\mathbf{M}^{(\lambda)}: \lambda>0\}$ be the matrix associated to $\omega_{\mathbf{M}}$. By \cite[Lemma 5.9]{compositionpaper} and \eqref{omega6} it follows that $\mathbf{M}^{(\lambda)}\approx\mathbf{M}^{(\kappa)}$, i.e. $\mathcal{M}_{\omega_{\mathbf{M}}}$ is constant. In this case we get $\mathbf{M}\equiv\mathbf{M}^{(1)}$ by definition of $\mathbf{M}^{(1)}$ and \cite[Prop. 3.2]{Komatsu73} (see also the proof of \cite[Thm. 6.4]{testfunctioncharacterization}):
     \begin{align*}
     M^{(1)}_p&:=\exp(\varphi^{*}_{\omega_{\mathbf{M}}}(p))=\exp(\sup_{y\ge 0}\{py-\omega_{\mathbf{M}}(e^y)\})=\sup_{y\ge 0}\exp(py-\omega_{\mathbf{M}}(e^y))
     \\&
     =\sup_{y\ge 0}\frac{\exp(py)}{\exp(\omega_{\mathbf{M}}(e^y))}=\sup_{t\ge 1}\frac{t^p}{\exp(\omega_{\mathbf{M}}(t))}=\sup_{t\ge 0}\frac{t^p}{\exp(\omega_{\mathbf{M}}(t))}=M_p.
     \end{align*}
     Note that by normalization we have $\omega_{\mathbf{M}}(t)=0$ for $0\le t\le 1$ which follows by the known integral representation formula for $\omega_M$, see \cite[1.8. III]{mandelbrojtbook} and also \cite[$(3.11)$]{Komatsu73}, and since $t^p\le 1$ for $0\le t\le 1$, $p\in\NN_0$ arbitrary.

     Consequently, $\mathbf{M}^{(\lambda)}\approx\mathbf{M}$ for all $\lambda>0$ and so $$\mathcal{S}_{\{\omega_{\mathbf{M}}\}}=\mathcal{S}_{\{\mathcal{M}_{\omega_{\mathbf{M}}}\}}=\mathcal{S}_{\{\mathbf{M}\}},\hspace{30pt}\mathcal{S}_{(\omega_{\mathbf{M}})}=\mathcal{S}_{(\mathcal{M}_{\omega_{\mathbf{M}}})}=\mathcal{S}_{(\mathbf{M})},$$
     as topological vector spaces (and the spaces can be defined equivalently by joint or separated growth at infinity).

\item[$(VI)$] By construction we have $\mu_{k_j}=2^j\nu_{k_j}$ for all $j\ge 3$ odd and $\mu_{k_j}=\frac{1}{j}\nu_{k_j}$ for all $j\ge 4$ even. Thus $$\liminf_{p\rightarrow\infty}\frac{\mu_p}{\nu_p}=0\quad\mbox{ and }\quad \limsup_{p\rightarrow\infty}\frac{\mu_p}{\nu_p}=+\infty.$$
    Now,
    $$\exists\;A\ge 1\;\forall\;p\in\NN:\;\;\;(M_p)^{1/p}\le\mu_p\le A(M_p)^{1/p}.$$
    In fact, the first estimate follows by log-convexity and normalization (see e.g. \cite[Lemma 2.0.4]{diploma}), the second one by moderate growth, e.g. see again \cite[Lemma 2.2]{whitneyextensionweightmatrix}. Consequently the sequences $(M_p^{1/p})_p$ and $(\mu_p)_p$ are comparable up to a constant. By \eqref{cond-fuzzy} the same is valid for $\mathbf{N}$ and so we have
    $$\liminf_{p\rightarrow\infty}\left(\frac{M_p}{N_p}\right)^{1/p}=0,\hspace{30pt}\limsup_{p\rightarrow\infty}\left(\frac{M_p}{N_p}\right)^{1/p}=+\infty.$$
    Hence, $\mathbf{M}$ and $\mathbf{N}$ are not comparable, which means that neither $\mathbf{M}\preceq\mathbf{N}$ nor $\mathbf{N}\preceq\mathbf{M}$  hold  (consequently, neither $\mathbf{M}\vartriangleleft\mathbf{N}$ nor $\mathbf{N}\vartriangleleft\mathbf{M}$, too).

\end{itemize}


\subsection{The critical example case}\label{oscillatingpreliminaries}
Now we treat a particular case. Namely, when the weight sequence $\mathbf{N}$ is the critical sequence $\mathbf{G}^{1/2}:=(p!^{1/2})_{p\in\NN}\in{\mathcal{LC}}$, and hence we look for a weight function $\omega\equiv\omega_{\mathbf{M}}$ which oscillates  around the critical weight function $\omega(t)=t^2$, related to the problem of nontriviality of $\mathcal{S}_{(\omega)}$ and $\mathcal{S}_{\{\omega\}}$. First, we summarize  some known facts:
\begin{itemize}
\item[$(i)$] The associated weight function $\omega_{\mathbf{G}^{1/2}}$ is equivalent to $t\mapsto t^2$, i.e. $\omega_{\mathbf{G}^{1/2}}(t)=O(t^2)$ and $t^2=O(\omega_{\mathbf{G}^{1/2}}(t))$ as $t\rightarrow\infty$.

\item[$(ii)$] The ultradifferentiable classes coincide, i.e. $\mathcal{E}_{\{\mathbf{G}^{1/2}\}}=\mathcal{E}_{\{\omega_{\mathbf{G}^{1/2}}\}}$, $\mathcal{E}_{(\mathbf{G}^{1/2})}=\mathcal{E}_{(\omega_{\mathbf{G}^{1/2}})}$ which follows from \cite[Lemma 5.9, Theorem 5.14]{compositionpaper} applied to the weight $\omega\equiv\omega_{\mathbf{G}^{1/2}}$. For this note that $\mathbf{G}^{1/2}$ satisfies \eqref{M2} and so equivalently (see \cite[Prop. 3.6]{Komatsu73}) the function $\omega_{\mathbf{G}^{1/2}}$ has \eqref{omega6}. Finally $\mathbf{G}^{1/2}$ satisfies $$\liminf_{j\rightarrow\infty}\frac{\mu_{Qj}}{\mu_j}=Q^{1/2}>1$$
 for all $Q\in\NN$, $Q\ge 2$, and so $\omega_{\mathbf{G}^{1/2}}$ has $(\alpha)$ which follows by \cite[Lemma 12 $(2)\Rightarrow(4)$]{BonetMeiseMelikhov07}.

\item[$(iii)$] Similarly, again by  \cite[Lemma 5.9, Theorem 5.14]{compositionpaper} and the analogous definition of the global classes, we get that $\mathcal{S}_{\{\mathbf{G}^{1/2}\}}=\mathcal{S}_{\{\omega_{\mathbf{G}^{1/2}}\}}$, $\mathcal{S}_{(\mathbf{G}^{1/2})}=\mathcal{S}_{(\omega_{\mathbf{G}^{1/2}})}$. Moreover, we can define these classes equivalently by using separate or joint growth at infinity which holds by \eqref{algebra} (implied by \eqref{M1}) and \eqref{M2}.\vspace{6pt}

\item[$(iv)$] However, note that formally \cite[Theorem 14]{BonetMeiseMelikhov07} cannot be applied directly to $\mathbf{G}^{1/2}$ in order to conclude $(ii)$ and/or $(iii)$ since their basic assumption
\beqs
\label{M0}
\exists c > 0: \quad (c(p + 1))^p\leq M_p, \qquad p\in\N_0,
\eeqs
is violated for $\mathbf{G}^{1/2}$. Recall that \eqref{M0} for $\mathbf{M}$ means that $\liminf_{p\rightarrow\infty}(M_p/p!)^{1/p}>0$, which yields $\omega_{\mathbf{M}}(t)=O(t)$, see \cite[Lemma 12 $(4)\Rightarrow(5)$]{BonetMeiseMelikhov07}. Note also that \eqref{M0} is slightly stronger than our assumption $\lim_{j\rightarrow\infty}(M_j)^{1/j}=\infty$.
\end{itemize}

Translating  into the notation of growth relations  \cite[Lemma~13]{nuclearglobal2}
(whose proof did not use assumption $(\beta)$) we have the following result:

\begin{lemma}\label{Lemma63revisited}
Let $\omega$ be a (general) weight function. Then
$$t\mapsto t^2\preceq\omega\Longleftrightarrow\omega(t)=O(t^2)\Longleftrightarrow\;\forall\;\lambda>0:\;\mathbf{G}^{1/2}\preceq \mathbf{W}^{(\lambda)},$$
and
$$t\mapsto t^2\vartriangleleft\omega\Longleftrightarrow\omega(t)=o(t^2)\Longleftrightarrow\;\forall\;\lambda>0:\;\mathbf{G}^{1/2}\vartriangleleft \mathbf{W}^{(\lambda)}.$$
\end{lemma}
Similarly, following the lines in the proof of \cite[Lemma 13]{nuclearglobal2} we obtain:

\begin{lemma}\label{Lemma63revisitedconverse}
Let $\omega$ be a (general) weight function. Then
$$\omega\preceq t\mapsto t^2\Longleftrightarrow t^2=O(\omega(t))\Longleftrightarrow\;\forall\;\lambda>0:\;\mathbf{W}^{(\lambda)}\preceq \mathbf{G}^{1/2},$$
and
$$\omega\vartriangleleft t\mapsto t^2\Longleftrightarrow t^2=o(\omega(t))\Longleftrightarrow\;\forall\;\lambda>0:\;\mathbf{W}^{(\lambda)}\vartriangleleft\mathbf{G}^{1/2}.$$
\end{lemma}

Note that these results follow also from \cite[Lemma 5.16, Corollary 5.17]{compositionpaper} and we can replace in all conditions ``$\forall\;\lambda>0$'' equivalently by ``$\exists\;\lambda>0$''.

Now, for $\mathbf{M}\equiv\mathbf{W}^{(\lambda)}$ we have that $\mathbf{M}$ and $\mathbf{G}^{1/2}$ are not comparable which means: Neither $\mathbf{M}\preceq\mathbf{G}^{1/2}$ nor $\mathbf{G}^{1/2}\preceq\mathbf{M}$ holds (hence neither $\mathbf{M}\vartriangleleft\mathbf{G}^{1/2}$ nor $\mathbf{G}^{1/2}\vartriangleleft\mathbf{M}$, too).

It follows from Lemmas \ref{Lemma63revisited} and \ref{Lemma63revisitedconverse}  that neither $\omega_{\mathbf{M}}\preceq t\mapsto t^2$ nor $t\mapsto t^2\preceq\omega_{\mathbf{M}}$ is valid (and hence neither $\omega_{\mathbf{M}}\vartriangleleft t\mapsto t^2$ nor $t\mapsto t^2\vartriangleleft\omega_{\mathbf{M}}$, too).

Finally, we mention that $\mathbf{M}$ does not satisfy the requirements of \cite{BonetMeiseMelikhov07} because their basic assumption \eqref{M0} is violated: since
\begin{equation}
	\liminf_{p\rightarrow\infty}\left(\frac{M_p}{p!^{1/2}}\right)^{1/p}=0
	\label{critical}
\end{equation}
 then also $\liminf_{p\rightarrow\infty}\left(\frac{M_p}{p!}\right)^{1/p}=0$.
 Moreover, from \eqref{critical} again we also get that the sequence $\mathbf{M}$ cannot satisfy the conditions in \cite[Prop.~3]{nuclearglobal2}. Hence, the spaces $\mathcal{S}_{(\mathbf{M})}$ and $\mathcal{S}_{\{\mathbf{M}\}}$ do not contain any Hermite function. Still, we do not know if these classes are nontrivial.
However, the existence of such an oscillating sequence is important in view of the following:
\begin{remark}{\rm
	\label{oscillatingconsequence}
Let $\omega$ be a given (general) weight function according to Definition \ref{defomega}. If $\omega(t)=o(t^2)$, then \cite[Cor.~3$(b)$]{nuclearglobal2} yields that $\mathcal{S}_{(\omega)}$ is nontrivial (at least all Hermite functions are contained in this class).

However, when $t^2=O(\omega(t))$ as $t\rightarrow\infty$, then we prove now that $\mathcal{S}_{(\omega)}=\{0\}$: First, for any $f\in\mathcal{S}_{(\omega)}$ we get
$$\forall\;\lambda>0:\;\;\;\sup_{x\in\RR^d}|f(x)e^{\lambda\omega(x)}|<\infty,\hspace{30pt}\sup_{\xi\in\RR^d}|\widehat{f}(\xi)e^{\lambda\omega(\xi)}|<\infty,$$
which gives, by the relation $t^2=O(\omega(t))$,
$$\sup_{x\in\RR^d}|f(x)e^{x^2/2}|<\infty,\hspace{30pt}\sup_{\xi\in\RR^d}|\widehat{f}(\xi)e^{\xi^2/2}|<\infty.$$
Now \cite[Corollary]{Hoe91} yields  $f\equiv 0$.\vspace{6pt}

Analogously, in the Roumieu case $\omega(t)=O(t^2)$ implies by \cite[Cor. 3 $(a)$]{nuclearglobal2} that $\mathcal{S}_{\{\omega\}}$ is nontrivial but $t^2=o(\omega(t))$ implies $\mathcal{S}_{\{\omega\}}=\{0\}$.}
\end{remark}


\section{Characterization of the inclusion relations of global ultradifferentiable classes}\label{characterizationsection}

In this section we characterize the inclusion relations of spaces of rapidly decreasing ultradifferentiable functions using the isomorphisms with sequence spaces obtained in
\cite{nuclearglobal2}.
Let us distinguish the various cases.


\subsection{The weight function case}\label{characterizationweightfunction}

In this case, for a weight function $\omega$ as in Definition~\ref{defomega},
we recall from \cite[Corollary 5]{nuclearglobal2}, the following isomorphisms,
 in the Roumieu case
\beqs
\label{Lambda1}
\quad\mathcal{S}_{\{\omega\}}(\RR^d)\cong
\Lambda_{\{\omega\}}:=\{&&\!\!\mathbf{c}=(c_{\alpha})_{\alpha\in\NN^d_0}\in\CC^{\NN^d}:\\
\nonumber
&&\exists j\in\NN:\
\|\mathbf{c}\|_{\omega,j}:=\sup_{\alpha\in\NN^d_0}|c_{\alpha}|e^{\frac{1}{j}\omega(\frac{\sqrt{\alpha}}{j})}<+\infty\},
\eeqs
where $\sqrt{\alpha}:=(\sqrt{\alpha_1},\dots,\sqrt{\alpha_d})$, and in the Beurling case,  under the additional assumption $\omega(t)=o(t^2)$ as $t\to+\infty$:
\beqs
\label{Lambda2}
\quad\mathcal{S}_{(\omega)}(\RR^d)\cong\Lambda_{(\omega)}:=\{&&\!\!\mathbf{c}=(c_{\alpha})_{\alpha\in\NN^d_0}\in\CC^{\NN^d}:\\
\nonumber
&&\forall j\in\NN:\
\|\mathbf{c}\|_{\omega,\frac{1}{j}}:=\sup_{\alpha\in\NN^d_0}|c_{\alpha}|e^{j\omega(\sqrt{\alpha}j)}<+\infty\}.
\eeqs

We can then prove the following:

\begin{theorem}\label{Gelfand-inclusion3new}
	Let $\omega$ and $\sigma$ be weight functions. Then the following are equivalent:
	\begin{itemize}
		\item[$(i)$] $\omega\preceq\sigma$, i.e. $\sigma(t)=O(\omega(t))$ as $t\rightarrow\infty$,

		\item[$(ii$)] $\mathcal{S}_{\{\omega\}}(\RR^d)\subseteq\mathcal{S}_{\{\sigma\}}(\RR^d)$ holds for all dimensions $d\in\NN$ with continuous inclusion.
	\end{itemize}
More precisely, the spaces in $(ii)$ can be defined by joint or separated growth at infinity; $(i)\Rightarrow(ii)$ is valid for general weight functions $\omega$ and $\sigma$ and for $(ii)\Rightarrow(i)$ only the inclusion for $d=1$ is required.
\end{theorem}

\begin{proof}
$(i)\Rightarrow(ii)$ follows by definition of the classes.

$(ii)\Rightarrow(i)$ By assumption and the isomorphism \eqref{Lambda1}
 we get  $\Lambda_{\{\omega\}}\cong\mathcal{S}_{\{\omega\}}(\RR^d)\subseteq\mathcal{S}_{\{\sigma\}}(\RR^d)\cong\Lambda_{\{\sigma\}}$.

We now use $(ii)$ for the case $d=1$. Consider the sequence $\mathbf{c}:=(c_k)_{k\in\NN_0}\in\CC^{\NN}$ defined by $c_k:=e^{-\omega(\sqrt{k})}$. It is clear that $\mathbf{c}\in\Lambda_{\{\omega\}}$ (choosing $j=1$) and so $\mathbf{c}\in\Lambda_{\{\sigma\}}$ as well. Thus
$$\exists\;l\in\NN\;\exists\;C\ge 1\;\forall\;k\in\NN_0:\;\;\;e^{-\omega(\sqrt{k})}=|c_k|\le Ce^{-\frac{1}{l}\sigma(\frac{\sqrt{k}}{l})},$$
which yields $\log(C)+\omega(\sqrt{k})\ge\frac{1}{l}\sigma(\frac{\sqrt{k}}{l})$ or equivalently $\sigma(\frac{\sqrt{k}}{l})\le l\log(C)+l\omega(\sqrt{k})$. We set $h:=\frac{\sqrt{k}}{l}$. Hence $\sigma(h)\le l\log(C)+l\omega(hl)\le l\log(C)+lC_1\omega(h)+lC_1$, by the iteration of property $(\alpha)$ for $\omega$. The number $n$ of iterations is only depending on given $l$, more precisely we choose $n\in\NN$ minimal to guarantee $l\le 2^n$. Note that all arising constants $C,C_1\ge 1$ and $l\in\NN$ are not depending on $k$ (hence not on $h$).

Finally let $t\in\RR$ with $\frac{\sqrt{k}}{l}<t<\frac{\sqrt{k+1}}{l}$ for some $k\in\NN$, then for $h:=\frac{\sqrt{k+1}}{l}$ we get
\begin{align*}
\sigma(t)&\le\sigma\left(\frac{\sqrt{k+1}}{l}\right)=\sigma(h)\le l\log(C)+l\omega(hl)\le l\log(C)+lC_1\omega\left(\frac{\sqrt{k+1}}{l}\right)+lC_1
\\&
\le l\log(C)+lC_1L\omega\left(\frac{\sqrt{k}}{l}\right)+lC_1L+lC_1\le l\log(C)+lC_1L\omega(t)+lC_1L+lC_1,
\end{align*}
with $L\ge 1$ denoting the constant arising in $(\alpha)$ for the weight $\omega$. More precisely, this condition implies $\omega(\frac{\sqrt{k+1}}{l})\le L\omega(\frac{\sqrt{k}}{l})+L$ for all $k\in\NN$ because $\frac{\sqrt{k+1}}{2l}\le\frac{\sqrt{k}}{l}$.
All arising constants are only depending on the given weights $\omega$ and $\sigma$ and not on $t\in\RR$. Thus we have shown $\limsup_{t\rightarrow\infty}\frac{\sigma(t)}{\omega(t)}<+\infty$, i.e. $\omega\preceq\sigma$.
\end{proof}

Now we treat the mixed case between the Roumieu- and Beurling-type classes.

\begin{theorem}\label{Gelfand-inclusionnew4}
	Let $\omega$ and $\sigma$ be weight functions with $\sigma(t)=o(t^2)$ as $t\rightarrow\infty$. Then the following are equivalent:
	\begin{itemize}
		\item[$(i)$] $\omega\vartriangleleft\sigma$, i.e. $\sigma(t)=o(\omega(t))$ as $t\rightarrow\infty$,

		\item[$(ii$)] $\mathcal{S}_{\{\omega\}}(\RR^d)\subseteq\mathcal{S}_{(\sigma)}(\RR^d)$ holds  for all dimensions $d\in\NN$ with continuous inclusion.
	\end{itemize}
More precisely, the spaces in $(ii)$ can be defined by joint or separated growth at infinity; $(i)\Rightarrow(ii)$ is valid for general weight functions $\omega$ and $\sigma$ and for $(ii)\Rightarrow(i)$ only the inclusion for $d=1$ is required.
\end{theorem}

\begin{proof}
$(i)\Rightarrow(ii)$ follows by definition of the classes.

$(ii)\Rightarrow(i)$ By assumption and the isomorphisms \eqref{Lambda1}-\eqref{Lambda2},
 we get $\Lambda_{\{\omega\}}\cong\mathcal{S}_{\{\omega\}}(\RR^d)\subseteq\mathcal{S}_{(\sigma)}(\RR^d)\cong\Lambda_{(\sigma)}$.
We use assumption $(ii)$ for the case $d=1$ and proceed similarly as before. Consider the sequence $\mathbf{c}:=(c_k)_{k\in\NN_0}\in\CC^{\NN}$ defined by $c_k:=e^{-\omega(\sqrt{k})}$. It is clear that $\mathbf{c}\in\Lambda_{\{\omega\}}$ (by choosing $j=1$) and so $\mathbf{c}\in\Lambda_{(\sigma)}$ as well. Thus
$$\forall\;l\in\NN\;\exists\;C\ge 1\;\forall\;k\in\NN_0:\;\;\;e^{-\omega(\sqrt{k})}=|c_k|\le Ce^{-l\sigma(\sqrt{k}l)},$$
which yields $\sigma(\sqrt{k})\le\sigma(\sqrt{k}l)\le\frac{\log(C)}{l}+\frac{1}{l}\omega(\sqrt{k})$.

Let $t\in\RR$ with $\sqrt{k}<t<\sqrt{k+1}$ for some $k\in\NN$, then we get
\begin{align*}
\sigma(t)&\le\sigma(\sqrt{k+1})\le\frac{\log(C)}{l}+\frac{1}{l}\omega(\sqrt{k+1})\le\frac{\log(C)}{l}+\frac{1}{l}L\omega(\sqrt{k})+\frac{L}{l}
\\&
\le\frac{\log(C)}{l}+\frac{L}{l}\omega(t)+\frac{L}{l}
\end{align*}
with $L\ge 1$ denoting the constant arising in $(\alpha)$ for $\omega$ (again we have used $\omega(\sqrt{k+1})\le L\omega(\sqrt{k})+L$ for all $k\in\NN_0$).
The arising constants $C$ and $L$ are only depending on the given weights $\omega$ and $\sigma$, but not on $t\in\RR$. Thus, as $l\rightarrow+\infty$, we have shown $\limsup_{t\rightarrow\infty}\frac{\sigma(t)}{\omega(t)}=0$, i.e. $\omega\vartriangleleft\sigma$.
\end{proof}

Finally we consider the Beurling case.

\begin{theorem}\label{Gelfand-inclusionbeurling}
	Let $\omega$ and $\sigma$ be weight functions with $\omega(t)=o(t^2)$, $\sigma=o(t^2)$ as $t\rightarrow\infty$. Then the following are equivalent:
	\begin{itemize}
		\item[$(i)$] $\omega\preceq\sigma$, i.e. $\sigma(t)=O(\omega(t))$ as $t\rightarrow\infty$,

		\item[$(ii$)] $\mathcal{S}_{(\omega)}(\RR^d)\subseteq\mathcal{S}_{(\sigma)}(\RR^d)$ holds  for all dimensions $d\in\NN$ with continuous inclusion.
	\end{itemize}
More precisely, the spaces in $(ii)$ can be defined by joint or separated growth control at infinity; $(i)\Rightarrow(ii)$ is valid for general weight functions $\omega$ and $\sigma$ and for $(ii)\Rightarrow(i)$ only the inclusion for $d=1$ is required.
\end{theorem}

\begin{proof}
$(i)\Rightarrow(ii)$ follows again by definition of the classes.

$(ii)\Rightarrow(i)$ By assumption and the isomorphism \eqref{Lambda2} we get now $\Lambda_{(\omega)}\cong\mathcal{S}_{(\omega)}(\RR^d)\subseteq\mathcal{S}_{(\sigma)}(\RR^d)\cong\Lambda_{(\sigma)}$ with continuous inclusion.
We use assumption $(ii)$ for the case $d=1$.
By the continuity of the inclusion we get
\begin{equation}\label{continclusion}
\forall\;j\in\NN\;\exists\;l\in\NN\;\exists\;C\ge 1\;\forall\;\mathbf{c}\in\Lambda_{(\omega)}:\;\;\;\|\mathbf{c}\|_{\sigma,\frac{1}{j}}\le C\|\mathbf{c}\|_{\omega,\frac{1}{l}}.
\end{equation}
For $i\in\NN_0$ we consider the sequence $\mathbf{c}^i$ defined by $c^i_k:=\delta_{i,k}$. It is clear that each $\mathbf{c}^i\in\Lambda_{(\omega)}$ because $\|\mathbf{c}^i\|_{\omega,\frac{1}{j}}=e^{j\omega(\sqrt{i}j)}<+\infty$ for all $i\in\NN_0$, $j\in\NN$. We apply \eqref{continclusion} to the case $j=1$ and the family $\mathbf{c}^i$, $i\in\NN_0$, and get
$$\exists\;l\in\NN\;\exists\;C\ge 1\forall\;k\in\NN_0:\;\;\;e^{\sigma(\sqrt{k})}\le Ce^{l\omega(\sqrt{k}l)}.$$
Consequently $\sigma(\sqrt{k})\le\log(C)+l\omega(\sqrt{k}l)$ follows. The iteration of property $(\alpha)$ for $\omega$ yields
$$\exists\;l\in\NN\;\exists\;C\ge 1\;\exists\;C_1\ge 1\;\forall\;k\in\NN_0:\;\;\;\sigma(\sqrt{k})\le\log(C)+lC_1\omega(\sqrt{k})+lC_1.$$
Note that all arising constants are not depending on $k$. Finally let $t\in\RR$ with $\sqrt{k}<t<\sqrt{k+1}$ for some $k\in\NN$. Then
\begin{align*}
\sigma(t)&\le\sigma(\sqrt{k+1})\le\log(C)+lC_1\omega(\sqrt{k+1})+lC_1\le\log(C)+lC_1L\omega(\sqrt{k})+lC_1L+lC_1
\\&
\le\log(C)+lC_1L\omega(t)+lC_1L+lC_1,
\end{align*}
again by applying $(\alpha)$ similarly as in the proof of Theorem \ref{Gelfand-inclusion3new} before. Since all arising constants are not depending on $t$ we have shown $\sigma(t)=O(\omega(t))$ as $t\rightarrow\infty$, i.e. $\omega\preceq\sigma$.
\end{proof}


\subsection{The general weight matrix case}\label{characterizationweightmatrix}
In this case, for a weight matrix $\mathcal{M}$
we recall the isomorphisms proved in \cite[Theorem 1]{nuclearglobal2}.
In the Roumieu case, if conditions \eqref{12L2R} and \eqref{M2'R} are satisfied, then
\beqs
\label{sequencematrixroumieu}
\mathcal{S}_{\{\mathcal{M}\}}(\RR^d)\cong
\Lambda_{\{\mathcal{M}\}}:=\{&&\!\!\mathbf{c}=(c_{\alpha})_{\alpha\in\NN^d_0}\in\CC^{\NN^d}:\\
\nonumber
&&\exists l\in\NN:\ \|\mathbf{c}\|_{\mathbf{M}^{(l)},l}:=\sup_{\alpha\in\NN^d_0}|c_{\alpha}|e^{\omega_{\mathbf{M}^{(l)}}(\frac{\sqrt{\alpha}}{l})}<+\infty\}.
\eeqs
Analogously in the Beurling case, if conditions \eqref{12L2B} and \eqref{M2'B} are satisfied, then
\beqs
\label{sequencematrixbeurling}
\mathcal{S}_{(\mathcal{M})}(\RR^d)\cong
\Lambda_{(\mathcal{M})}:=\{&&\!\!\mathbf{c}=(c_{\alpha})_{\alpha\in\NN^d_0}\in\CC^{\NN^d}:\\
\nonumber
&&\forall l\in\NN:\ \|\mathbf{c}\|_{\mathbf{M}^{(1/l)},\frac{1}{l}}:=\sup_{\alpha\in\NN^d_0}|c_{\alpha}|e^{\omega_{\mathbf{M}^{(1/l)}}(\sqrt{\alpha}l)}<+\infty\}.
\eeqs

We can thus prove similar results as in \S\ref{characterizationweightfunction}.
We start with the Roumieu case.

\begin{theorem}\label{Gelfand-inclusion2new}
	Let $\mathcal{M}:=\{\mathbf{M}^{(\lambda)}: \lambda>0\}$ and $\mathcal{N}:=\{\mathbf{N}^{(\lambda)}: \lambda>0\}$ be given weight matrices and consider the following assertions:
	\begin{itemize}
		\item[$(i)$] $\mathcal{M}\{\preceq\}\mathcal{N}$,
		
		\item[$(ii$)] $\mathcal{S}_{\{\mathcal{M}\}}(\RR^d)\subseteq\mathcal{S}_{\{\mathcal{N}\}}(\RR^d)$ holds  with continuous inclusion.
	\end{itemize}
Then we get the following: $(i)\Rightarrow(ii)$ is valid for all dimensions $d\in\NN$ and the classes in $(ii)$ can be defined by joint or separated growth at infinity. If $(ii)$ holds for the case $d=1$ and both matrices are standard log-convex with \eqref{12L2R} and \eqref{M2'R}, then $(ii)\Rightarrow(i)$ is valid, too.
\end{theorem}

\begin{proof}
Again, $(i)\Rightarrow(ii)$ follows by the definition of the spaces.

$(ii)\Rightarrow(i)$ We use the inclusion for the dimension $d=1$ and so the matrices consist only of sequences $\mathbf{M}^{(\lambda)},\mathbf{N}^{(\lambda)}\in{\mathcal{LC}}$.

By the assumptions on $\mathcal{M}$ and $\mathcal{N}$ we can apply the isomorphism
\eqref{sequencematrixroumieu} and so $(ii)$ yields $\Lambda_{\{\mathcal{M}\}}\cong\mathcal{S}_{\{\mathcal{M}\}}(\RR^d)\subseteq\mathcal{S}_{\{\mathcal{N}\}}(\RR^d)\cong\Lambda_{\{\mathcal{N}\}}$.

We consider the sequence $\mathbf{c}:=(c_k)_{k\in\NN_0}\in\CC^{\NN}$ defined by $c_k:=e^{-\omega_{\mathbf{M}^{(j)}}(\frac{\sqrt{k}}{j})}$ with $j\in\NN$, $j\ge 2$, arbitrary but from now on fixed. So $\mathbf{c}\in\Lambda_{\{\mathcal{M}\}}$ follows by choosing $l=j$ in \eqref{sequencematrixroumieu} and this yields $\mathbf{c}\in\Lambda_{\{\mathcal{N}\}}$ as well. Thus
$$\forall\;j\in\NN\;\exists\;l\in\NN\;\exists\;C\ge 1\;\forall\;k\in\NN_0:\;\;\;e^{-\omega_{\mathbf{M}^{(j)}}(\frac{\sqrt{k}}{j})}=|c_k|\le Ce^{-\omega_{\mathbf{N}^{(l)}}(\frac{\sqrt{k}}{l})},$$
which implies $\log(C)+\omega_{\mathbf{M}^{(j)}}(\sqrt{k})\ge\log(C)+\omega_{\mathbf{M}^{(j)}}(\frac{\sqrt{k}}{j})\ge\omega_{\mathbf{N}^{(l)}}(\frac{\sqrt{k}}{l})$.

Let now $t\in\RR$ with $\sqrt{k}<t<\sqrt{k+1}$ for some $k\in\NN$. Then we estimate
\begin{align*}
\omega_{\mathbf{N}^{(l)}}\left(\frac{t}{l}\right)&\le\omega_{\mathbf{N}^{(l)}}\left(\frac{\sqrt{k+1}}{l}\right)\le\log(C)+\omega_{\mathbf{M}^{(j)}}\left(\frac{\sqrt{k+1}}{j}\right)\le\log(C)+\omega_{\mathbf{M}^{(j)}}(\sqrt{k})
\\&
\le\log(C)+\omega_{\mathbf{M}^{(j)}}(t).
\end{align*}
Here we have used that $\frac{\sqrt{k+1}}{j}\le\sqrt{k}$ is valid for any $k\in\NN$ when $j\ge 2$ and that each $\omega_{\mathbf{M}^{(j)}}$ is increasing. Finally, if $0<t<1$, then $\omega_{\mathbf{N}^{(l)}}\left(\frac{t}{l}\right)\le\omega_{\mathbf{N}^{(l)}}\left(\frac{1}{l}\right)$. Consequently, by enlarging the constant $C$ if necessary, so far we have shown
$$\forall\;j\in\NN,\;j\ge 2,\;\exists\;l\in\NN\;\exists\;C\ge 1\;\forall\;t\ge 0:\;\;\;\omega_{\mathbf{N}^{(l)}}\left(\frac{t}{l}\right)\le\log(C)+\omega_{\mathbf{M}^{(j)}}(t).$$
We use this estimate and the fact that each sequence belonging to the matrices is log-convex and normalized. Hence, by \eqref{Prop32Komatsu} we get for all $p\in\NN_0$:
\begin{align*}
M^{(j)}_p&=\sup_{t\ge 0}\frac{t^p}{\exp(\omega_{\mathbf{M}^{(j)}}(t))}\le C\sup_{t\ge 0}\frac{t^p}{\exp(\omega_{\mathbf{N}^{(l)}}(\frac{t}{l}))}=C\sup_{s\ge 0}\frac{(sl)^p}{\exp(\omega_{\mathbf{N}^{(l)}}(s))}=Cl^pN^{(l)}_p,
\end{align*}
which proves
$$\forall\;j\in\NN,\;j\ge 2,\;\exists\;l\in\NN:\;\;\;\mathbf{M}^{(j)}\preceq\mathbf{N}^{(l)}$$
and so $\mathcal{M}\{\preceq\}\mathcal{N}$ is verified. Note that the assumption $j\ge 2$ is not restricting the generality in our considerations since we are dealing with Roumieu type spaces.
\end{proof}

Next we treat the mixed situation between the Roumieu case and the Beurling case.

\begin{theorem}\label{Gelfand-inclusion2newmixed}
	Let $\mathcal{M}:=\{\mathbf{M}^{(\lambda)}: \lambda>0\}$ and $\mathcal{N}:=\{\mathbf{N}^{(\lambda)}: \lambda>0\}$ given weight matrices and consider the following assertions:
	\begin{itemize}
		\item[$(i)$] $\mathcal{M}\vartriangleleft\mathcal{N}$,
		
		\item[$(ii$)] $\mathcal{S}_{\{\mathcal{M}\}}(\RR^d)\subseteq\mathcal{S}_{(\mathcal{N})}(\RR^d)$ holds  with continuous inclusion.
	\end{itemize}
Then we get the following: $(i)\Rightarrow(ii)$ is valid, for all dimensions $d\in\NN$ and the classes in $(ii)$ can be defined by joint or separated growth at infinity. If $(ii)$ holds for the case $d=1$ and if both matrices are standard log-convex and $\mathcal{M}$ does have \eqref{12L2R} and \eqref{M2'R}, whereas $\mathcal{N}$ is required to satisfy \eqref{12L2B} and \eqref{M2'B}, then $(ii)\Rightarrow(i)$ is valid, too.
\end{theorem}

\begin{proof}
Again, $(i)\Rightarrow(ii)$ follows by the definition of the spaces.\vspace{6pt}

$(ii)\Rightarrow(i)$ We use this inclusion for $d=1$.
By the assumptions on $\mathcal{M}$ and $\mathcal{N}$
and the  isomorphisms \eqref{sequencematrixroumieu}-\eqref{sequencematrixbeurling} we have that $(ii)$ yields $\Lambda_{\{\mathcal{M}\}}\cong\mathcal{S}_{\{\mathcal{M}\}}(\RR^d)\subseteq\mathcal{S}_{(\mathcal{N})}(\RR^d)\cong\Lambda_{(\mathcal{N})}$.
As in the previous proof we consider the sequence $\mathbf{c}:=(c_k)_{k\in\NN_0}\in\CC^{\NN}$ defined by $c_k:=e^{-\omega_{\mathbf{M}^{(j)}}(\frac{\sqrt{k}}{j})}$ with $j\in\NN$, $j\ge 2$, arbitrary but from now on fixed. So $\mathbf{c}\in\Lambda_{\{\mathcal{M}\}}$ by choosing $l=j$ and now the assumption yields $\mathbf{c}\in\Lambda_{(\mathcal{N})}$ as well. Thus we obtain
$$\forall\;j\in\NN,\;j\ge 2,\;\forall\;l\in\NN\;\exists\;C\ge 1\;\forall\;k\in\NN_0:\;\;\;e^{-\omega_{\mathbf{M}^{(j)}}(\frac{\sqrt{k}}{j})}=|c_k|\le Ce^{-\omega_{\mathbf{N}^{(1/l)}}(\sqrt{k}l)},$$
which gives $\log(C)+\omega_{\mathbf{M}^{(j)}}(\sqrt{k})\ge\log(C)+\omega_{\mathbf{M}^{(j)}}(\frac{\sqrt{k}}{j})\ge\omega_{\mathbf{N}^{(1/l)}}(\sqrt{k}l)$ and note that the arising constant $C$ is depending on $l$ and $j$.

Let now $t\in\RR$ with $\sqrt{k}<t<\sqrt{k+1}$ for some $k\in\NN$. Then
\begin{align*}
\omega_{\mathbf{N}^{(1/l)}}(tl)&\le\omega_{\mathbf{N}^{(1/l)}}(\sqrt{k+1}l)\le\log(C)+\omega_{\mathbf{M}^{(j)}}\left(\frac{\sqrt{k+1}}{j}\right)\le\log(C)+\omega_{\mathbf{M}^{(j)}}(\sqrt{k})
\\&
\le\log(C)+\omega_{\mathbf{M}^{(j)}}(t),
\end{align*}
as in the proof of Theorem \ref{Gelfand-inclusion2new}. Finally, if $0<t<1$, then $\omega_{\mathbf{N}^{(1/l)}}(tl)\le\omega_{\mathbf{N}^{(1/l)}}(l)$. Consequently, by enlarging the constant $C$ if necessary, so far we have shown
$$\forall\;j\in\NN\;j\ge 2,\;\forall\;l\in\NN\;\exists\;C\ge 1\;\forall\;t\ge 0:\;\;\;\omega_{\mathbf{N}^{(1/l)}}(tl)\le\log(C)+\omega_{\mathbf{M}^{(j)}}(t).$$
We use this estimate and the fact that each sequence belonging to the matrices is log-convex and normalized, hence by \eqref{Prop32Komatsu} we get for all $p\in\NN_0$ and $i\le l$:
\begin{align*}
M^{(j)}_p&=\sup_{t\ge 0}\frac{t^p}{\exp(\omega_{\mathbf{M}^{(j)}}(t))}\le C\sup_{t\ge 0}\frac{t^p}{\exp(\omega_{\mathbf{N}^{(1/l)}}(tl))}=C\sup_{s\ge 0}\frac{(s/l)^p}{\exp(\omega_{\mathbf{N}^{(1/l)}}(s))}=C\frac{1}{l^p}N^{(1/l)}_p
\\&
\le C\frac{1}{l^p}N^{(1/i)}_p.
\end{align*}
This estimate proves $\mathbf{M}^{(j)}\vartriangleleft\mathbf{N}^{(1/i)}$ for all $i,j\in\NN$, $j\ge 2$: Let $i$ and $j\ge 2$ be arbitrary but fixed, then we get by the previous computations that $\left(\frac{M^{(j)}_p}{N^{(1/i)}_p}\right)^{1/p}\le C_l^{1/p}\frac{1}{l}$ for all $l\ge i$ and $p\in\NN$. Assumption $j\ge 2$ is not restricting since the matrix $\mathcal{M}$ is related to Roumieu-type conditions and small indices can be omitted without changing the corresponding function class. Thus we have verified $\mathcal{M}\vartriangleleft\mathcal{N}$.
\end{proof}

Finally we treat the general weight matrix case in the Beurling-type setting.

\begin{theorem}\label{Gelfand-inclusionbeurlingmatrix}
	Let $\mathcal{M}:=\{\mathbf{M}^{(\lambda)}: \lambda>0\}$ and $\mathcal{N}:=\{\mathbf{N}^{(\lambda)}: \lambda>0\}$ be given and consider the following assertions:
	\begin{itemize}
		\item[$(i)$] $\mathcal{M}(\preceq)\mathcal{N}$,
		
		\item[$(ii$)] $\mathcal{S}_{(\mathcal{M})}(\RR^d)\subseteq\mathcal{S}_{(\mathcal{N})}(\RR^d)$ holds  with continuous inclusion.
	\end{itemize}
Then we get the following: $(i)\Rightarrow(ii)$ is valid for all dimensions $d\in\NN$ and the classes in $(ii)$ can be defined by joint or separated growth at infinity. If $(ii)$ holds for the case $d=1$ both matrices are standard log-convex with \eqref{12L2B} and \eqref{M2'B}, then $(ii)\Rightarrow(i)$ is valid, too.
\end{theorem}

\begin{proof}
Again, $(i)\Rightarrow(ii)$ follows by the definition of the spaces.\vspace{6pt}

$(ii)\Rightarrow(i)$ We use this inclusion for $d=1$.
By the assumptions on $\mathcal{M}$ and $\mathcal{N}$ and the isomorphism \eqref{sequencematrixbeurling} we have that $(ii)$ yields $\Lambda_{(\mathcal{M})}\cong\mathcal{S}_{(\mathcal{M})}(\RR^d)\subseteq\mathcal{S}_{(\mathcal{N})}(\RR^d)\cong\Lambda_{(\mathcal{N})}$  with continuous inclusion.
Now we proceed analogously as in the proof of Theorem \ref{Gelfand-inclusionbeurling} before.
By the continuity of the inclusion we get
\begin{equation}\label{continclusionmatrix}
\forall\;j\in\NN\;\exists\;l\in\NN\;\exists\;C\ge 1\;\forall\;\mathbf{c}\in\Lambda_{(\mathcal{M})}:\;\;\;\|\mathbf{c}\|_{\mathbf{N}^{(1/j)},\frac{1}{j}}\le C\|\mathbf{c}\|_{\mathbf{M}^{(1/l)},\frac{1}{l}}.
\end{equation}
For $i\in\NN_0$ we consider again the sequence $\mathbf{c}^i$ defined by $c^i_k:=\delta_{i,k}$. It is clear that each $\mathbf{c}^i\in\Lambda_{(\mathcal{M})}$ because $\|\mathbf{c}^i\|_{\mathbf{M}^{(1/j)},\frac{1}{j}}=e^{\omega_{\mathbf{M}^{(1/j)}}(\sqrt{i}j)}<+\infty$ for all $i\in\NN_0$ and $j\in\NN$. We apply \eqref{continclusionmatrix} to the family $\mathbf{c}^i$, $i\in\NN_0$, and get
$$\forall\;j\in\NN\;\exists\;l\in\NN\;\exists\;C\ge 1\forall\;k\in\NN_0:\;\;\;e^{\omega_{\mathbf{N}^{(1/j)}}(\sqrt{k}j)}\le Ce^{\omega_{\mathbf{M}^{(1/l)}}(\sqrt{k}l)},$$
consequently $\omega_{\mathbf{N}^{(1/j)}}(\sqrt{k}j)\le\log(C)+\omega_{\mathbf{M}^{(1/l)}}(\sqrt{k}l)\le\log(C)+\omega_{\mathbf{M}^{(1/l)}}(\sqrt{k}2l)$ follows because each $\omega_{\mathbf{M}^{(j)}}$ is increasing.

Let now $t\in\RR$ with $\sqrt{k}<t<\sqrt{k+1}$ for some $k\in\NN$. Then we estimate by
\begin{align*}
\omega_{\mathbf{N}^{(1/j)}}(tj)&\le\omega_{\mathbf{N}^{(1/j)}}(\sqrt{k+1}j)\le\log(C)+\omega_{\mathbf{M}^{(1/l)}}(\sqrt{k+1}l)\le\log(C)+\omega_{\mathbf{M}^{(1/l)}}(\sqrt{k}2l)
\\&
\le\log(C)+\omega_{\mathbf{M}^{(1/l)}}(t2l).
\end{align*}
Here we have used that $\sqrt{k+1}\le2\sqrt{k}$.

If $t\in\RR$ with $0<t<1$, then $\omega_{\mathbf{N}^{(1/j)}}(tj)\le\omega_{\mathbf{N}^{(1/j)}}(j)$. Consequently, by enlarging the constant $C$ if necessary, so far we have shown
$$\forall\;j\in\NN\;\exists\;l\in\NN\;\exists\;C\ge 1\;\forall\;t\ge 0:\;\;\;\omega_{\mathbf{N}^{(1/j)}}(t)\le\log(C)+\omega_{\mathbf{M}^{(1/l)}}(t2l/j).$$
Finally, by using this and again \eqref{Prop32Komatsu} we get for all $p\in\NN_0$:
\begin{align*}
N^{(1/j)}_p&=\sup_{t\ge 0}\frac{t^p}{\exp(\omega_{\mathbf{N}^{(1/j)}}(t))}\ge\frac{1}{C}\sup_{t\ge 0}\frac{t^p}{\exp(\omega_{\mathbf{M}^{(1/l)}}(t2l/j))}=\frac{1}{C}\sup_{s\ge 0}\frac{(sj/(2l))^p}{\exp(\omega_{\mathbf{M}^{(1/l)}}(s))}
\\&
=\frac{1}{C}\left(\frac{j}{2l}\right)^pM^{(1/l)}_p,
\end{align*}
which proves $\mathbf{M}^{(1/l)}\preceq\mathbf{N}^{(1/j)}$ and so $\mathcal{M}(\preceq)\mathcal{N}$ is verified.
\end{proof}

\subsection{The single weight sequence case}\label{characterizationweightsquence}
It is straight-forward to obtain the analogous results for Theorems \ref{Gelfand-inclusion2new}, \ref{Gelfand-inclusion2newmixed} and \ref{Gelfand-inclusionbeurlingmatrix} in the single weight sequence setting and we get the following characterization:

\begin{theorem}\label{Gelfandweightsequence}
Let $\mathbf{M},\mathbf{N}\in{\mathcal{LC}}$ be given such that both satisfy \eqref{M2prime}.

\begin{itemize}
\item[$(I)$] Let $\mathbf{M}$ and $\mathbf{N}$ satisfy \eqref{12L2R} and consider the following assertions:
\begin{itemize}
\item[$(i)$] $\mathbf{M}\preceq\mathbf{N}$,

\item[$(ii)$] $\mathcal{S}_{\{\mathbf{M}\}}(\RR^d)\subseteq\mathcal{S}_{\{\mathbf{N}\}}(\RR^d)$ with continuous inclusion.
\end{itemize}

\item[$(II)$] Let $\mathbf{M}$  and  $\mathbf{N}$ satisfy  \eqref{12L2R} and \eqref{12L2B} respectively, and consider the following assertions:
\begin{itemize}
\item[$(i)$] $\mathbf{M}\vartriangleleft\mathbf{N}$,

\item[$(ii)$] $\mathcal{S}_{\{\mathbf{M}\}}(\RR^d)\subseteq\mathcal{S}_{(\mathbf{N})}(\RR^d)$ with continuous inclusion.
\end{itemize}

\item[$(III)$] Let $\mathbf{M}$ and $\mathbf{N}$ satisfy \eqref{12L2B} and consider the following assertions:
\begin{itemize}
\item[$(i)$] $\mathbf{M}\preceq\mathbf{N}$,

\item[$(ii)$] $\mathcal{S}_{(\mathbf{M})}(\RR^d)\subseteq\mathcal{S}_{(\mathbf{N})}(\RR^d)$ with continuous inclusion.
\end{itemize}
\end{itemize}
Then all implications $(i)\Rightarrow(ii)$ are valid for arbitrary sequences $\mathbf{M},\mathbf{N}\in\RR_{>0}^{\NN}$, for all dimensions $d\in\NN$ and the spaces can be defined in terms of a joint or separated growth at infinity. If $(ii)$ holds for $d=1$, then the implications $(ii)\Rightarrow(i)$ are valid, too.
\end{theorem}

\begin{remark}\label{anisoremark}
If we consider the {\itshape isotropic} setting, so we set $M^{(\lambda)}_{\alpha}:=M^{(\lambda)}_{|\alpha|}$ (resp. $M_{\alpha}:=M_{|\alpha|}$) for any $\lambda>0$ and $\alpha\in\NN^d_0$, then in all results from Section \ref{characterizationweightmatrix} and Section \ref{characterizationweightsquence} we have that $(ii)\Rightarrow(i)$ is valid if $(ii)$ holds for some/any dimension $d\in\NN$. Similarly this applies to the results listed in Section~\ref{alternativesection} as well.
For the analogous results in the anisotropic setting we refer to \cite{BJOS-cvxminorant}.
\end{remark}


\section{Comparison of classes defined by weight sequences and weight functions}\label{comparisonofGSclasses}
Gathering the information from the previous section we are now able to prove the following results which are analogous to the statements obtained in  \cite{BonetMeiseMelikhov07} and \cite{compositionpaper}
for the spaces $\mathcal{E}_{\{\mathbf{M}\}}, \mathcal{E}_{(\mathbf{M})},
\mathcal{E}_{\{\omega\}}, \mathcal{E}_{(\omega)}$ (cf. also
 \cite[Remark 4]{nuclearglobal2}).


\begin{theorem}\label{comparsion}
Let $\omega$ be a weight function with associated weight matrix $\mathcal{M}_{\omega}:=\{\mathbf{W}^{(\lambda)}: \lambda>0\}$. Then the following are equivalent:
\begin{itemize}
\item[$(i)$] $\omega$  satisfies \eqref{omega6},

\item[$(ii)$] there  exists $\mathbf{M}\in{\mathcal{LC}}$ such that:
\begin{itemize}
\item[$(ii.1)$] $\mathbf{M}$ satisfies \eqref{M2}, hence \eqref{M2prime};

\item[$(ii.2)$] $\mathbf{M}$ satisfies \eqref{12L2R};

\item[$(ii.3)$] $\omega_{\mathbf{M}}$ satisfies $(\alpha)$;

\item[$(ii.4)$] for any $d\in\NN$, if $M_{\alpha}:=M_{|\alpha|}$, $\alpha\in\NN^d_0$, then    \begin{equation}\label{comparisonequ}
    \mathcal{S}_{\{\omega\}}(\RR^d)=\mathcal{S}_{\{\mathbf{M}\}}(\RR^d),
    \end{equation}
   as topological vector spaces.
\end{itemize}
\end{itemize}
The analogous result holds true for the Beurling case as well, when considering (in addition) $\omega(t)=o(t^2)$ as $t\rightarrow\infty$ for the weight function $\omega$ and condition \eqref{12L2B}  instead of \eqref{12L2R}.

In both cases we can take  $\mathbf{M}\equiv\mathbf{W}^{(\lambda)}$ for some/each $\lambda>0$ in $(ii)$.

When the above equivalence holds true the space can be defined by separate or joint growth at infinity.

\end{theorem}

\begin{proof}
We will only treat the Roumieu case explicitly. The Beurling case follows analogously.

{\itshape The Roumieu case} $(i)\Rightarrow(ii)$: First, by \cite[Lemma 5.9 $(5.11)$]{compositionpaper} we get that $\mathcal{M}_{\omega}$ is constant, more precisely $\mathcal{M}_{\omega}\{\approx\}\mathbf{W}^{(\lambda)}$ for some/each $\lambda>0$. Thus, by definition of the spaces and Proposition \ref{Proposition61} we get as topological vector spaces for all $d\in\NN$
$$\mathcal{S}_{\{\omega\}}(\RR^d)=\mathcal{S}_{\{\mathcal{M}_{\omega}\}}(\RR^d)=\mathcal{S}_{\{\mathbf{W^{(\lambda)}}\}}(\RR^d),\qquad\forall\;\lambda>0.$$
Condition $\mathbf{W}^{(\lambda)}\in{\mathcal{LC}}$ is clear by definition. Moreover, \cite[Corollary 5.8 $(2)$]{compositionpaper} yields that some/each $\mathbf{W}^{(\lambda)}$
satisfies \eqref{M2}, hence \eqref{M2prime} as well. Also \eqref{12L2R} for some/each $\mathbf{W}^{(\lambda)}$ follows by \cite[Lemma 13 $(a)$]{nuclearglobal2} applied to $r=1/2$ which can be done by assumption $(\beta)$ on $\omega$ and by \cite[Proposition 3 $(a)\Rightarrow(b)$]{nuclearglobal2}.

Finally, that $\omega_{\mathbf{W}^{(\lambda)}}$ satisfies $(\alpha)$ for some/each $\lambda>0$ follows by the fact that $(\alpha)$ holds true for $\omega$ by assumption, by \cite[Lemma 5.7]{compositionpaper} and because this condition is clearly stable under equivalence of weight functions.

$(ii)\Rightarrow(i)$: First, we want to show that the matrix $\mathcal{M}_{\omega}$ is constant, i.e. $\mathbf{W}^{(\lambda)}\approx\mathbf{W}^{(\kappa)}$ for all $\lambda,\kappa>0$.

By Proposition \ref{Proposition61} and assumption $(ii.4)$ we get, as topological vector spaces,
$$\mathcal{S}_{\{\mathcal{M}_{\omega}\}}(\RR^d)=\mathcal{S}_{\{\omega\}}(\RR^d)=\mathcal{S}_{\{\mathbf{M}\}}(\RR^d).$$
Now Theorem \ref{Gelfand-inclusion2new} applied to the inclusion $\mathcal{S}_{\{\mathcal{M}_{\omega}\}}(\RR)\subseteq\mathcal{S}_{\{\mathbf{M}\}}(\RR)$ and to $\mathcal{M}\equiv\mathcal{M}_{\omega}$, $\mathcal{N}\equiv\{\mathbf{M}\}$, yields $\mathcal{M}_{\omega}\{\preceq\}\mathbf{M}$. By the converse inclusion $\mathcal{S}_{\{\mathbf{M}\}}(\RR)\subseteq\mathcal{S}_{\{\mathcal{M}_{\omega}\}}(\RR)$ and Theorem \ref{Gelfand-inclusion2new} applied to $\mathcal{M}\equiv\{\mathbf{M}\}$ and $\mathcal{N}\equiv\mathcal{M}_{\omega}$ we get $\mathbf{M}\{\preceq\}\mathcal{M}_{\omega}$ as well.

Recall that we can apply this characterizing result since $\omega$ is assumed to be a weight function and because of $(ii.1)$ and $(ii.2)$ for $\mathbf{M}$.

Summarizing, so far we have shown $\mathcal{M}_{\omega}\{\approx\}\mathbf{M}$ which clearly implies that $\mathcal{M}_{\omega}$ is constant. Then \cite[Lemma 5.9 $(5.11)$]{compositionpaper} yields \eqref{omega6} for $\omega$ and $\mathbf{M}\approx\mathbf{W}^{(\lambda)}$ for some/any $\lambda>0$ follows.

\end{proof}

Conversely, in the next result we start with a weight sequence, however the required arguments for the proof are the same as before.

\begin{theorem}\label{comparsion1}
Let $\mathbf{M}\in{\mathcal{LC}}$ be given and set $M_{\alpha}:=M_{|\alpha|}$ for any $\alpha\in\NN^d_0$. Assume that:
\begin{itemize}
\item[$(a)$] $\mathbf{M}$ satisfies \eqref{12L2R};

\item[$(b)$] $\mathbf{M}$ satisfies \eqref{M2prime};

\item[$(c)$] $\omega_{\mathbf{M}}$ satisfies $(\alpha)$.
\end{itemize}
Then the following are equivalent:
\begin{itemize}
\item[$(i)$] $\mathbf{M}$ satisfies \eqref{M2};

\item[$(ii)$] there exists a weight function $\omega$ satisfying \eqref{omega6} such that  for all $d\in\NN$
    \begin{equation}\label{comparisonequ1}
    \mathcal{S}_{\{\omega\}}(\RR^d)=\mathcal{S}_{\{\mathbf{M}\}}(\RR^d)
    \end{equation}
    as topological vector spaces.
\end{itemize}
The analogous result holds true for the Beurling case as well when $\mathbf{M}$ satisfies
  \eqref{12L2B} (instead of  \eqref{12L2R})  and (in addition) $\omega(t)=o(t^2)$ as $t\rightarrow\infty$.

In both cases we can take  the weight function $\omega\equiv\omega_{\mathbf{M}}$ in $(ii)$.
When the above equivalence holds true the space can be defined by separate or joint growth at infinity.
\end{theorem}

\begin{proof}
Again we only treat the Roumieu case.

$(i)\Rightarrow(ii)$ We consider the weight function $\omega_{\mathbf{M}}$. As mentioned in Section \ref{construction}, the basic assumptions to be a weight function, $(\gamma)$ and $(\delta)$ hold true automatically by definition and $(\alpha)$ follows by assumption $(c)$; \eqref{M2} implies \eqref{omega6}, see \cite[Proposition 3.6]{Komatsu73}. The choice $\beta=0$ in \eqref{12L2R} for $\mathbf{M}$ and the proof of $(6.6)$ in \cite[Lemma 13 $(a)$]{nuclearglobal2} applied to $r=1/2$ and to $\mathbf{W}^{(\lambda)}\equiv\mathbf{M}$  imply $(\beta)$ for $\omega_{\mathbf{M}}$.

Thus $\omega_{\mathbf{M}}$ is a weight function as required for $(ii)$. Let $\mathcal{M}_{\omega_{\mathbf{M}}}:=\{\mathbf{M}^{(\lambda)}: \lambda>0\}$ be the matrix associated to
$\omega_{\mathbf{M}}$. By \cite[Lemma 5.9 $(5.11)$]{compositionpaper} we get that this matrix is constant and since $\mathbf{M}\equiv\mathbf{M}^{(1)}$ (see $(V)$ in Section \ref{construction}) we have $\mathbf{M}^{(\lambda)}\approx\mathbf{M}$ for all $\lambda>0$. This shows \eqref{comparisonequ1} for $\omega\equiv\omega_{\mathbf{M}}$ by taking into account \cite[Proposition 5]{nuclearglobal2}.

$(ii)\Rightarrow(i)$ We follow the proof of $(ii)\Rightarrow(i)$ in Theorem \ref{comparsion} by applying Theorem \ref{Gelfand-inclusion2new} twice which can be done by the assumptions on $\omega$ and $\mathbf{M}$. By \eqref{omega6} the associated matrix $\mathcal{M}_{\omega}$ is constant (see \cite[Lemma 5.9 $(5.11)$]{compositionpaper}), some/each
$\mathbf{W}^{(\lambda)}$ satisfies \eqref{M2} (see \cite[Proposition 3.6]{Komatsu73}) and finally $\mathbf{W}^{(\lambda)}\approx\mathbf{M}$ for some/each $\lambda>0$ holds true. Since \eqref{M2} is obviously stable under the equivalence relation $\approx$, the proof is complete.
\end{proof}

\begin{remark}
Recently, in \cite[Thm. 3.1]{subaddlike} the requirement that $\omega_{\mathbf{M}}$ has $(\alpha)$, arising in $(ii.3)$ in Theorem \ref{comparsion} and in assumption $(c)$ in Theorem \ref{comparsion1}, has been characterized in terms of $\mathbf{M}\in{\mathcal{LC}}$ by the following condition:
$$\exists\;L\in\NN:\;\;\;\liminf_{p\rightarrow\infty}\frac{(M_{Lp})^{1/(Lp)}}{(M_p)^{1/p}}>1.$$
\end{remark}


\section{Alternative assumptions for the characterization of the inclusion relations}
\label{alternativesection}

The aim of this section is to present alternative proofs for the characterizations of the inclusion relations for Gelfand-Shilov classes. More precisely, we are not using results from \cite{nuclearglobal2} in the classes $\mathcal{S}$ (with matrix/function weights),  but following ideas generally used in the ultradifferentiable setting $\mathcal{E}$ (with matrix/function weights).
 In this case our assumptions are slightly different and we have to involve the property of {\itshape non-quasianalyticity}, which was  not required in section~\ref{characterizationsection}.
 We refer to Remark~\ref{nonquasiremark} to compare the distinct assumptions of sections
 \ref{characterizationsection} and \ref{alternativesection}: stronger conditions on the first weight, but
 weaker conditions on the second weight.
 Another difference with respect to
 Section~\ref{characterizationsection} is that here the Roumieu and the Beurling cases require different techniques.
 So let us start by the Roumieu case.

\subsection{The Roumieu case}
Let $\mathbf{M}\in{\mathcal{LC}}$ be given. We recall (e.g. see \cite[Lemma 2.9]{compositionpaper}): There exists a function $\theta_{\mathbf{M}}$ belonging to the ultradifferentiable class $\mathcal{E}_{\{\mathbf{M}\}}(\RR)$ and such that $|\theta_{\mathbf{M}}^{(j)}(0)|\ge M_j$ for all $j\in\NN$. In fact $\theta_{\mathbf{M}}$ belongs to the global ultradifferentiable class since the estimates are valid on whole $\RR$, more precisely:
$$\exists\;C,h>0\;\forall\;j\in\NN_0\;\forall\;x\in\RR:\;\;\;|\theta_{\mathbf{M}}^{(j)}(x)|\le Ch^jM_j.$$

In \cite{compositionpaper} such a function has been called a {\itshape characteristic function}. We can assume $\theta_{\mathbf{M}}$ to be real- or complex-valued (see the proof of \cite[Lemma 2.9]{compositionpaper}). Those functions are in some sense optimal in the ultradifferentiable classes of Roumieu-type.

Please note, that $\theta_{\mathbf{M}}$ cannot belong to the Beurling-type class $\mathcal{E}_{(\mathbf{M})}(\RR)$. Such functions are useful to characterize the inclusion relations of (global/local) ultradifferentiable function classes in terms of growth relations of weight sequences/functions or even matrices, see \cite[Prop. 2.12, Prop. 4.6, Cor. 5.17]{compositionpaper}.\vspace{6pt}

However, for our purposes we need that $\theta_{\mathbf{M}}\in\mathcal{S}_{\{\mathbf{M}\}}(\RR)$, $\mathbf{M}\in{\mathcal{LC}}$. To this aim we assume that $\mathbf{M}$ is {\itshape non-quasianalytic}, i.e.
\begin{equation}\label{nonquasianalytic}
\sum_{k\ge 1}\frac{1}{\mu_k}=\sum_{k\ge 1}\frac{M_{k-1}}{M_k}<\infty\Longleftrightarrow\sum_{k\ge 1}\frac{1}{(M_k)^{1/k}}<\infty.
\end{equation}
For the last equivalence we refer to \cite[Lemma~4.1]{Komatsu73}.

Then we can proceed as follows.
First, by the well-known Denjoy-Carleman-theorem we obtain that both the classes $\mathcal{D}_{\{\mathbf{M}\}}$ and $\mathcal{D}_{(\mathbf{M})}$ are non-trivial (e.g. see \cite[Theorem 4.2]{Komatsu73}).

Hence, in this situation we can define
$$\psi_{\mathbf{M}}:=\theta_{\mathbf{M}}\cdot\phi,$$
with $\phi\in\mathcal{D}_{\{\mathbf{M}\}}$ having $\phi^{(j)}(0)=\delta_{j,0}$ (Kronecker delta). For the existence of such a test function we refer to the proof of \cite[Thm. 2.2]{petzsche}. Concerning the support of $\phi$ we do not make any restriction.\vspace{6pt}

First, $\psi_{\mathbf{M}}\in\mathcal{D}_{\{\mathbf{M}\}}\subseteq\mathcal{S}_{\{\mathbf{M}\}}$: $\psi_{\mathbf{M}}$ obviously has compact support $K\ni 0$ with $\supp(\psi_{\mathbf{M}})\subseteq\supp(\phi)=K$ and moreover both $\theta_{\mathbf{M}}$ and $\phi$ admit growth control expressed in terms of the weight sequence $\mathbf{M}$. Hence for all $j\in\NN_0$ and $x\in K$:
\begin{align*}
|\psi^{(j)}_{\mathbf{M}}(x)|&=\left|\sum_{k=0}^{j}\binom{j}{k}\theta_{\mathbf{M}}^{(j-k)}(x)\cdot\phi^{(k)}(x)\right|\le\sum_{k=0}^{j}\binom{j}{k}|\theta_{\mathbf{M}}^{(j-k)}(x)||\phi^{(k)}(x)|
\\&
\le \sum_{k=0}^{j}\binom{j}{k}C_1h_1^{j-k}M_{j-k}C_2h_2^kM_k\le CM_j\sum_{k=0}^{j}\binom{j}{k}h_1^kh_2^{j-k}=Ch^jM_j,
\end{align*}
with $C:=C_1C_2$ and $h:=h_1+h_2$. Here we have used that $M_kM_{j-k}\le M_j$ for all $j,k\in\NN_0$ with $0\le k\le j$ which follows by log-convexity and normalization (e.g. see \cite[Lemma 2.0.6]{diploma}). Hence $\psi_{\mathbf{M}}\in\mathcal{E}_{\{\mathbf{M}\}}$ is shown, more precisely since $\supp(\psi_{\mathbf{M}})\subseteq K$ we have $\psi_{\mathbf{M}}\in\mathcal{D}_{\{\mathbf{M}\}}\subseteq\mathcal{S}_{\{\mathbf{M}\}}$.\vspace{6pt}

Second, by the product rule we have $\psi^{(j)}_{\mathbf{M}}(0)=\sum_{0\le k\le j}\binom{j}{k}\theta_{\mathbf{M}}^{(k)}(0)\cdot\phi^{(j-k)}(0)=\theta^{(j)}_{\mathbf{M}}(0)$ since only the summand with $k=j$ is ``surviving''. Thus $|\psi^{(j)}_{\mathbf{M}}(0)|\ge M_j$ for all $j\in\NN$ as well.

With this preparation we are able to prove the first main statement.

\begin{theorem}\label{Gelfand-inclusion1}
Let $\mathbf{M},\mathbf{N}\in\RR_{>0}^{\NN}$ be given and consider the following assertions:
\begin{itemize}
	\item[$(i)$] $\mathbf{M}\preceq\mathbf{N}$,
	
	\item[$(ii$)] $\mathcal{S}_{\{\mathbf{M}\}}(\RR^d)\subseteq\mathcal{S}_{\{\mathbf{N}\}}(\RR^d)$ is valid with continuous inclusion.
\end{itemize}

Then $(i)\Rightarrow(ii)$ is valid for all dimensions $d\in\NN$ and the classes in $(ii)$ can be defined by joint or separated growth at infinity. If $\mathbf{M}\in{\mathcal{LC}}$ such that \eqref{nonquasianalytic} is valid, if $(ii)$ holds for the case $d=1$, then $(ii)\Rightarrow(i)$ is valid, too.
\end{theorem}

\begin{proof}
$(i)\Rightarrow(ii)$ follows immediately by the definition of the spaces.\vspace{6pt}

$(ii)\Rightarrow(i)$ By the assumptions on the weights and the case $d=1$ we have $\psi_{\mathbf{M}}\in\mathcal{S}_{\{\mathbf{M}\}}(\RR)\subseteq\mathcal{S}_{\{\mathbf{N}\}}(\RR)$. So we can find $C,h>0$ such that $|\psi_{\mathbf{M}}^{(j)}(x)|\le Ch^jN_j$ for all $x\in\RR$ and $j\in\NN_0$. Consequently, for $x=0$, we get $M_j\le|\psi_{\mathbf{M}}^{(j)}(0)|\le Ch^jN_j$ for all $j\in\NN_0$, which implies $\mathbf{M}\preceq\mathbf{N}$.
\end{proof}




This result can be immediately used to get the same for the general weight matrix case. A standard log-convex matrix $\mathcal{M}$ is called {\itshape Roumieu non-quasianalytic,} if there exists some $\lambda_0>0$ such that $\mathbf{M}^{(\lambda_0)}$ is non-quasianalytic.

\begin{theorem}\label{Gelfand-inclusion2}
	Let $\mathcal{M}:=\{\mathbf{M}^{(\lambda)}: \lambda>0\}$, $\mathcal{N}:=\{\mathbf{N}^{(\lambda)}: \lambda>0\}$ be arbitrary and consider the following assertions:
	\begin{itemize}
		\item[$(i)$] $\mathcal{M}\{\preceq\}\mathcal{N}$,
		
		\item[$(ii$)] $\mathcal{S}_{\{\mathcal{M}\}}(\RR^d)\subseteq\mathcal{S}_{\{\mathcal{N}\}}(\RR^d)$ is valid with continuous inclusion.
	\end{itemize}
Then $(i)\Rightarrow(ii)$ is valid for all dimensions $d\in\NN$ and the classes in $(ii)$ can be defined by joint or separated growth at infinity. If $\mathcal{M}$ is standard log-convex and Roumieu non-quasianalytic and if $(ii)$ holds for the case $d=1$, then $(ii)\Rightarrow(i)$ is valid, too.
\end{theorem}

\begin{proof}
$(i)\Rightarrow(ii)$ follows again by the definition of the spaces.\vspace{6pt}

For $(ii)\Rightarrow(i)$ we use this inclusion for $d=1$. Since $\mathcal{M}$ is standard log-convex, for each given index $\lambda>0$ we can find a characteristic function $\theta_{\mathbf{M}^{(\lambda)}}$. Since $\mathcal{M}$ is Roumieu non-quasianalytic, we can assume that each $\mathbf{M}^{(\lambda)}$ is non-quasianalytic as well and so $\mathcal{D}_{\{\mathbf{M}^{(\lambda)}\}}$ is non-trivial. This can be achieved by ``not considering'' all possible quasianalytic sequences $M^{(\lambda)}$, $\lambda<\lambda_0$, which by definition does not change the according function classes. Treating the function $\psi_{\mathbf{M}^{(\lambda)}}$ we see that
$$\psi_{\mathbf{M}^{(\lambda)}}\in\mathcal{S}_{\{\mathbf{M}^{(\lambda)}\}}(\RR)\subseteq\mathcal{S}_{\{\mathcal{M}\}}(\RR)\subseteq\mathcal{S}_{\{\mathcal{N}\}}(\RR),$$
hence there exists some $\kappa>0$ and $C,h>0$ such that $|\psi_{\mathbf{M}^{(\lambda)}}^{(j)}(x)|\le Ch^jN^{(\kappa)}_j$ for all $x\in\RR$ and $j\in\NN_0$. Consequently, for $x=0$, we get $M^{(\lambda)}_j\le|\psi_{\mathbf{M}^{(\lambda)}}^{(j)}(0)|\le Ch^jN^{(\kappa)}_j$ for all $j\in\NN_0$, which immediately implies $\mathbf{M}^{(\lambda)}\preceq\mathbf{N}^{(\kappa)}$.
\end{proof}

The weight function case is reduced to Theorem \ref{Gelfand-inclusion2} by using the associated weight matrix $\mathcal{M}_{\omega}$ and for this we recall:

Given any (general) weight function $\omega$ we get that $\mathcal{M}_{\omega}$ is always standard log-convex (see Lemma \ref{Lemma61}). We call $\omega$ {\em non-quasianalytic}, if
\begin{equation}\label{omeganonquasi}
\int_{1}^{\infty}\frac{\omega(t)}{t^2}dt<+\infty.
\end{equation}
It is known, see \cite[Cor. 4.8]{testfunctioncharacterization}, that $\omega$ is non-quasianalytic if and only if some/each $\mathbf{W}^{(\lambda)}$ is non-quasianalytic, i.e. $\mathcal{M}_{\omega}$ is Roumieu non-quasianalytic.\vspace{6pt}

By applying Theorem \ref{Gelfand-inclusion2} to the associated weight matrices we obtain:

\begin{theorem}\label{Gelfand-inclusion3}
	Let $\omega$ be a non-quasianalytic weight function and $\sigma$ be a (general) weight function. Then the following are equivalent:
	\begin{itemize}
		\item[$(i)$] $\omega\preceq\sigma$,

		\item[$(ii$)] $\mathcal{S}_{\{\omega\}}(\RR^d)\subseteq\mathcal{S}_{\{\sigma\}}(\RR^d)$ is valid for all $d\in\NN$ with continuous inclusion.
	\end{itemize}
The spaces in $(ii)$ can be defined by joint or separated growth at infinity and $(i)\Rightarrow(ii)$ is valid for general weight functions $\omega$ and $\sigma$. For $(ii)\Rightarrow(i)$ only the inclusion for $d=1$ is required.
\end{theorem}

\begin{proof}
We denote by $\mathcal{M}_{\omega}$ and $\mathcal{M}_{\sigma}$ the associated weight matrices. First recall that $\omega\preceq\sigma$ if and only if $\mathcal{M}_{\omega}\{\preceq\}\mathcal{M}_{\sigma}$, which follows by \cite[Prop. 4.6 $(1)$, Thm. 5.14 $(2)$, Lemma 5.16 $(1)$, Cor. 5.17 $(1)$]{compositionpaper}.

Second, by \cite[Proposition 5]{nuclearglobal2} we get that $$\mathcal{S}_{\{\mathcal{M}_{\omega}\}}(\RR^d)=\mathcal{S}_{\{\omega\}}(\RR^d)\subseteq\mathcal{S}_{\{\sigma\}}(\RR^d)=\mathcal{S}_{\{\mathcal{M}_{\sigma}\}}(\RR^d)$$ holds true for all dimensions $d\in\NN$ and the equalities hold as topological vector spaces.

Since both $\mathcal{M}_{\omega}$ and $\mathcal{M}_{\sigma}$ are standard log-convex and $\mathcal{M}_{\omega}$ is Roumieu non-quasianalytic, we can apply Theorem \ref{Gelfand-inclusion2} in order to conclude.
\end{proof}

\begin{remark}\label{nonquasiremark}
The different assumptions on the weights in Theorems \ref{Gelfand-inclusion3new} and \ref{Gelfand-inclusion3} should be compared:

\begin{itemize}
\item[$(i)$] On the one hand, in Theorem \ref{Gelfand-inclusion3} for the weight $\sigma$ we do not require necessarily condition $(\beta)$ in Definition \ref{defomega} since for \cite[Proposition 5]{nuclearglobal2} this requirement is not needed.

\item[$(ii)$] However, for the weight $\omega$ we require the assumption of non-quasianalyticity,
which is stronger than condition $(\beta)$ because $\int_{t}^{\infty}\frac{\omega(u)}{u^2}du\ge\omega(t)\int_{t}^{\infty}\frac{1}{u^2}du=\frac{\omega(t)}{t}$ and so we get that $\eqref{omeganonquasi}\Rightarrow\omega(t)=o(t)\Rightarrow\omega(t)=o(t^2)\Rightarrow(\beta)$.
\end{itemize}
\end{remark}

Following analogous arguments and replacing for the larger weight structure the Roumieu- by the Beurling-type classes, it is immediate to obtain the next result:

\begin{theorem}\label{Gelfand-inclusion4}
We have the following:

\begin{itemize}
\item[$(I)$] Let $\mathbf{M}\in{\mathcal{LC}}$ be non-quasianalytic and $\mathbf{N}\in\RR_{>0}^{\NN}$ be arbitrary. Consider the following assertions:
\begin{itemize}
	\item[$(i)$] $\mathbf{M}\vartriangleleft\mathbf{N}$,
	
	\item[$(ii$)] $\mathcal{S}_{\{\mathbf{M}\}}(\RR^d)\subseteq\mathcal{S}_{(\mathbf{N})}(\RR^d)$ with continuous inclusion.
\end{itemize}

\item[$(II)$] Let $\mathcal{M}:=\{\mathbf{M}^{(\lambda)}: \lambda>0\}$ be standard log-convex and Roumieu non-quasianalytic, moreover let $\mathcal{N}:=\{\mathbf{N}^{(\lambda)}: \lambda>0\}$ be arbitrary. Consider the following assertions:
	\begin{itemize}
		\item[$(i)$] $\mathcal{M}\vartriangleleft\mathcal{N}$,
		
		\item[$(ii$)] $\mathcal{S}_{\{\mathcal{M}\}}(\RR^d)\subseteq\mathcal{S}_{(\mathcal{N})}(\RR^d)$ with continuous inclusion.
	\end{itemize}

\item[$(III)$] Let $\omega$ be a non-quasianalytic weight function and $\sigma$ be a (general) weight function. Consider the following assertions:
	\begin{itemize}
		\item[$(i)$] $\omega\vartriangleleft\sigma$,

		\item[$(ii$)] $\mathcal{S}_{\{\omega\}}(\RR^d)\subseteq\mathcal{S}_{(\sigma)}(\RR^d)$ with continuous inclusion.
	\end{itemize}
\end{itemize}
Then $(i)\Rightarrow(ii)$ is valid for arbitrary weight sequences (resp. matrices, weight functions) and the spaces in $(ii)$ can be defined by joint or separated growth at infinity. For the implications $(ii)\Rightarrow(i)$ only the inclusion for $d=1$ is required.
\end{theorem}



\subsection{The Beurling case}\label{Beurlingsquaresection}
We call a standard log-convex weight matrix $\mathcal{M}=\{\mathbf{M}^{(\lambda)}: \lambda>0\}$ {\itshape Beurling non-quasianalytic,} when for all $\lambda>0$ the sequence $\mathbf{M}^{(\lambda)}$ is non-quasianalytic. This definition is justified by \cite[Thm. 4.1, Sect. 4.6]{testfunctioncharacterization}: A countable intersection of non-quasianalytic ultradifferentiable classes (with totally ordered weight sequences) is still non-quasianalytic. So, if $\mathcal{M}$ is standard log-convex and {\itshape Beurling non-quasianalytic,} then $$\mathcal{D}_{(\mathcal{M})}:=\mathcal{D}\cap\mathcal{E}_{(\mathcal{M})}=\mathcal{D}\cap\bigcap_{\lambda>0}\mathcal{E}_{(\mathbf{M}^{(\lambda)})}=\mathcal{D}\cap\bigcap_{n\in\NN}\mathcal{E}_{(\mathbf{M}^{(1/n)})}$$ is non-trivial. More precisely, by \cite[Prop. 4.7 $(i)$, Prop. 4.4]{testfunctioncharacterization} we know that there exists $L\in{\mathcal{LC}}$ such that $L$ is non-quasianalytic and $L\vartriangleleft\mathcal{M}$. (Recall that for huge intersections this statement will fail in general, so an uncountable intersection of non-quasianalytic classes will be quasianalytic in general.) Finally, by the comments given in the previous section we see that for any given weight function $\omega$ the associated matrix $\mathcal{M}_{\omega}$ will be Beurling non-quasianalytic if and only if $\omega$ is non-quasianalytic, i.e. in the weight function approach both matrix notions of non-quasianalyticity do coincide.\vspace{6pt}

Let us now consider the following Beurling-type condition:
\begin{equation}\label{Beurlingsquare}
\forall\;\lambda>0\;\exists\;\kappa>0\;\exists\;A\ge 1\;\forall\;p\in\NN_0:\;\;\;(M^{(\kappa)}_p)^2\le A^pM^{(\lambda)}_p.
\end{equation}
It is immediate to see that $\mathcal{M}=\{\mathbf{M}\}$, $\mathbf{M}\in{\mathcal{LC}}$, can never satisfy \eqref{Beurlingsquare} because this would yield $\sup_{p\ge 1}(M_p)^{1/p}<+\infty$.

Moreover, if $\mathcal{M}$ is standard log-convex and satisfies \eqref{12L2B}, then for all $\lambda>0$ there exist $H>0$ and $B>0$ such that $j^{j/2}\leq BH^{j}M^{(\lambda)}_{j}$ for all $j\in\NN_0$, i.e. $\mathbf{G}^{1/2}\preceq\mathbf{M}^{(\lambda)}$. Thus it is immediate to see that any standard log-convex matrix having \eqref{Beurlingsquare} and \eqref{12L2B} is Beurling non-quasianalytic.
Condition
\eqref{12L2B} should be considered as a standard assumption when dealing with $\mathcal{S}_{(\mathcal{M})}$, e.g. see \cite[Prop. 3]{nuclearglobal2}.\vspace{6pt}

Now we are ready to state the following result which is analogous to Theorem \ref{Gelfand-inclusion2}.

\begin{theorem}\label{Beurlingprop}
Let $\mathcal{M}$ and $\mathcal{N}$ be arbitrary and consider the following assertions:
	\begin{itemize}
		\item[$(i)$] $\mathcal{M}(\preceq)\mathcal{N}$,
		
		\item[$(ii$)] $\mathcal{S}_{(\mathcal{M})}(\RR^d)\subseteq\mathcal{S}_{(\mathcal{N})}(\RR^d)$ is valid with continuous inclusion.
	\end{itemize}
Then $(i)\Rightarrow(ii)$ and the spaces in $(ii)$ can be defined by joint or separated growth at infinity. If both matrices are standard log-convex and if we assume that $\mathcal{M}$ is Beurling non-quasianalytic and satisfies \eqref{Beurlingsquare} and if $(ii)$ holds for the case $d=1$, then $(ii)\Rightarrow(i)$ is valid, too.
\end{theorem}

\begin{proof}
$(i)\Rightarrow(ii)$ is again clear by the definition of the spaces.\vspace{6pt}

$(ii)\Rightarrow(i)$ We follow the  ideas of the proof given in \cite[Prop. 4.6 $(2)$]{compositionpaper} which is based on techniques developed in \cite[Sect. 2]{Bruna} (for the single weight sequence case).

By the continuous inclusion $\mathcal{S}_{(\mathcal{M})}(\RR)\subseteq\mathcal{S}_{(\mathcal{N})}(\RR)$ we get the following:
\begin{equation}\label{Beurlingpropequ}
\forall\;\lambda>0\;\forall\;h>0\;\exists\;\kappa>0\;\exists\;C,\;h_1>0\;\forall\;f\in\mathcal{S}_{(\mathcal{M})}(\RR):\;\;\;\|f\|_{\infty,\mathbf{N}^{(\lambda)},h}\le C\|f\|_{\infty,\mathbf{M}^{(\kappa)},h_1}.
\end{equation}
We will apply \eqref{Beurlingpropequ} for $h=1$ and to the following family of functions. For each $a>0$, arbitrary but from now on fixed, we consider a function $\phi_a\in\mathcal{D}_{(\mathcal{M})}$ with $\supp(\phi_a)\subseteq[-a,a]$ and $\phi_a^{(j)}(0)=\delta_{j,0}$: The existence of such functions follows again by \cite[Thm. 2.2]{petzsche}, more precisely we take $\phi_a\in\mathcal{D}_{\{L\}}\subseteq\mathcal{D}_{(\mathcal{M})}$ with $L\vartriangleleft\mathcal{M}$ denoting the non-quasianalytic sequence mentioned before. (Here we use the fact that $\mathcal{M}$ is Beurling non-quasianalytic.) Moreover, for $t\ge 0$ and $x\in\RR$ we set $f_t(x):=\exp(itx)$ and finally
$$g_{a,t}(x):=f_t(x)\cdot\phi_a(x).$$
First, let us prove that $g_{a,t}\in\mathcal{S}_{(\mathcal{M})}$. Let $t>0$ be fixed (the case $t=0$ is trivial), then note that for all $h,\lambda>0$ (small) there exists some $C_{h,\lambda}\ge 1$ (large) such that $|f_t^{(j)}(x)|=t^j\le C_{h,\lambda}h^jM^{(\lambda)}_j$ for all $j\in\NN_0$ because $\lim_{j\rightarrow\infty}(M^{(\lambda)}_j)^{1/j}=+\infty$ is valid. This proves that $f_t\in\mathcal{E}_{(\mathcal{M})}$ and the estimates hold globally on the whole $\RR$.

We have that $\supp(g_{a,t})\subseteq[-a,a]$ for all $t\ge 0$. We fix $t\ge 0$ and $a>0$ and estimate for all $j\in\NN_0$, $\lambda,h>0$ and $x\in[-a,a]$ as follows:
\beqsn
|g_{a,t}^{(j)}(x)|=&&\left|\sum_{k=0}^{j}\binom{j}{k}f_t^{(j-k)}(x)\cdot\phi_a^{(k)}(x)\right|\le\sum_{k=0}^{j}\binom{j}{k}|f_t^{(j-k)}(x)||\phi_a^{(k)}(x)|\\
=&&\sum_{k=0}^{j}\binom{j}{k}t^{j-k}|\phi_a^{(k)}(x)|
\le\sum_{k=0}^{j}\binom{j}{k}C'_{h,\lambda}h^{j-k}M^{(\lambda)}_{j-k}\cdot C_{h,\lambda}h^kM^{(\lambda)}_k\\
&&\le C''_{h,\lambda}M^{(\lambda)}_j\sum_{k=0}^{j}\binom{j}{k}h^{j-k}h^k=C''_{h,\lambda}(2h)^jM^{(\lambda)}_j,
\eeqsn
because by normalization and log-convexity we have $M^{(\lambda)}_{j-k}M^{(\lambda)}_k\le M^{(\lambda)}_j$ for all $j,k\in\NN_0$ with $k\le j$ and for each $\lambda$.

Thus $g_{a,t}\in\mathcal{D}_{(\mathcal{M})}\subseteq\mathcal{S}_{(\mathcal{M})}$ and we are able to apply \eqref{Beurlingpropequ} to this family (with $h=1$ in \eqref{Beurlingpropequ} and set $a:=h_1(\le 1)$).\vspace{6pt}

According to the index $\kappa$ arising in \eqref{Beurlingpropequ}, by applying \eqref{Beurlingsquare} we get an index $\kappa_1$ and $A\ge 1$ such that $(M^{(\kappa_1)}_p)^2\le A^pM^{(\kappa)}_p$ for any $p\in\NN_0$.

Using this preparation we start now by estimating the right-hand side in \eqref{Beurlingpropequ} for all $t\ge 1$ (and by using the seminorms with the joint growth control at infinity):
\beqs
\nonumber
C\|g_{h_1,t}\|_{\infty,\mathbf{M}^{(\kappa)},h_1}=&&C\sup_{j,k\in\NN_0}\sup_{x\in\RR}\frac{|x^jg_{h_1,t}^{(k)}(x)|}{h_1^{j+k}M^{(\kappa)}_{j+k}}=C\sup_{j,k\in\NN_0}\sup_{x\in[-h_1,h_1]}\frac{|x^jg_{h_1,t}^{(k)}(x)|}{h_1^{j+k}M^{(\kappa)}_{j+k}}
\\
\nonumber
\le&& C\sup_{j,k\in\NN_0}\sup_{x\in[-h_1,h_1]}\frac{h_1^j|g_{h_1,t}^{(k)}(x)|}{h_1^{j+k}M^{(\kappa)}_{j+k}}=C\sup_{j,k\in\NN_0}\sup_{x\in[-h_1,h_1]}\frac{|g_{h_1,t}^{(k)}(x)|}{h_1^{k}M^{(\kappa)}_{j+k}}\\
\label{underbracestar}
\leq&&
CC_{h_1,\kappa_1}\sup_{j,k\in\NN_0}\frac{t^k(1+h_1)^kM^{(\kappa_1)}_k}{h_1^kM^{(\kappa)}_{j+k}}
\le CC_{h_1,\kappa_1}\sup_{k\in\NN_0}\frac{(2t)^kM^{(\kappa_1)}_k}{h_1^kM^{(\kappa)}_k}\\
\nonumber
\le&& CC_{h_1,\kappa_1}\sup_{k\in\NN_0}\frac{(2tA)^k}{h_1^kM^{(\kappa_1)}_k}=CC_{h_1,\kappa_1}\exp\left(\omega_{\mathbf{M}^{(\kappa_1)}}(2At/h_1)\right).
\eeqs
For the estimate of \eqref{underbracestar} we argued as follows: Since $\phi_{h_1}\in\mathcal{D}_{(\mathcal{M})}\subseteq\mathcal{S}_{(\mathcal{M})}(\RR)$, we get
\beqsn
|g_{h_1,t}^{(k)}(x)|\le&&\sum_{l=0}^k\binom{k}{l}t^l|\phi_{h_1}^{(k-l)}(x)|\le\sum_{l=0}^k\binom{k}{l}t^lC_{h_1,\kappa_1}h_1^{k-l}M^{(\kappa_1)}_{k-l}\\
\le&& C_{h_1,\kappa_1}M^{(\kappa_1)}_kt^k\sum_{l=0}^k\binom{k}{l}h_1^{k-l}=C_{h_1,\kappa_1}t^kM^{(\kappa_1)}_k(1+h_1)^k.
\eeqsn
Note that by log-convexity and normalization $M^{(\kappa_1)}_{k-l}\le M^{(\kappa_1)}_k$, i.e. each sequence is increasing and since we are dealing with the Beurling case we will have $0<h_1\le 1$ (small). Moreover note that we have estimated by $\frac{1}{M^{(\kappa)}_{j+k}}\le\frac{1}{M^{(\kappa)}_{k}}$ for any $j,k\in\NN_0$ and any index $\kappa$, so if we would consider separated growth control at infinity we would estimate at this step by $\frac{1}{M^{(\kappa)}_{k}M^{(\kappa)}_{j}}\le\frac{1}{M^{(\kappa)}_{k}}$.\vspace{6pt}

We continue now with the left-hand side in \eqref{Beurlingpropequ} and get
\begin{align*}
\|g_{h_1,t}\|_{\infty,\mathbf{N}^{(\lambda)},1}&=\sup_{j,k\in\NN_0}\sup_{x\in\RR}\frac{|x^jg_{h_1,t}^{(k)}(x)|}{N^{(\lambda)}_{j+k}}\underbrace{\ge}_{x=0=j}\sup_{k\in\NN_0}\frac{|g_{h_1,t}^{(k)}(0)|}{N^{(\lambda)}_{k}}=\sup_{k\in\NN_0}\frac{t^k}{N^{(\lambda)}_{k}}=\exp\left(\omega_{\mathbf{N}^{(\lambda)}}(t)\right),
\end{align*}
and the estimate is precisely the same when considering separated growth control at infinity by normalization. Summarizing, we have shown that \eqref{Beurlingpropequ} yields the following:
$$\forall\;\lambda>0\;\exists\;\kappa_1>0\;\exists\;C,A,h_1>0\;\forall\;t\ge 1:\;\;\;\exp\left(\omega_{\mathbf{N}^{(\lambda)}}(t)\right)\le C\exp\left(\omega_{\mathbf{M}^{(\kappa_1)}}(2At/h_1)\right).$$
Since $\mathbf{N}^{(\lambda)}\in{\mathcal{LC}}$ (for any $\lambda>0$) we get $\omega_{\mathbf{N}^{(\lambda)}}(t)=0$ for all $t\in[0,1]$, recalling $(V)$ in
Section~\ref{construction}. So the above estimate is valid for all $t\ge 0$. Finally, we are applying \eqref{Prop32Komatsu} and obtain for any $p\in\NN_0$ that
\beqsn
N^{(\lambda)}_p=&&\sup_{t\ge 0}\frac{t^p}{\exp(\omega_{\mathbf{N}^{(\lambda)}}(t))}\ge\frac{1}{C}\sup_{t\ge 0}\frac{t^p}{\exp(\omega_{\mathbf{M}^{(\kappa_1)}}(2At/h_1))}\\
=&&\frac{1}{C}\left(\frac{h_1}{2A}\right)^p\sup_{s\ge 0}\frac{s^p}{\exp(\omega_{\mathbf{M}^{(\kappa_1)}}(s))}
=\frac{1}{C}\left(\frac{h_1}{2A}\right)^pM^{(\kappa_1)}_p,
\eeqsn
which implies $\mathcal{M}(\preceq)\mathcal{N}$ and finishes the proof.
\end{proof}

Again, by involving the associated weight matrices, we can transfer Theorem \ref{Beurlingprop} to the weight function situation:

\begin{theorem}\label{Beurlingpropweightfct}
Let $\omega$ be a weight function with
\begin{equation}\label{om7}
\exists\;H>0:\;\;\;\omega(t^2)=O(\omega(Ht)),\;\;\;t\rightarrow+\infty,
\end{equation}
and $\sigma$ be a (general) weight function. Then the following are equivalent:
	\begin{itemize}
		\item[$(i)$] $\omega\preceq\sigma$,

		\item[$(ii$)] $\mathcal{S}_{(\omega)}(\RR^d)\subseteq\mathcal{S}_{(\sigma)}(\RR^d)$ holds true for all dimensions $d\in\NN$ with continuous inclusion.
	\end{itemize}
Again the spaces in $(ii)$ can be defined by joint or separated growth at infinity and $(i)\Rightarrow(ii)$ is valid for general weight functions $\omega$ and $\sigma$. For $(ii)\Rightarrow(i)$ only the inclusion for $d=1$ is required.
\end{theorem}

\begin{proof}
We denote by $\mathcal{M}_{\omega}$ and $\mathcal{M}_{\sigma}$ the associated weight matrices. First recall that $\omega\preceq\sigma$ if and only if $\mathcal{M}_{\omega}(\preceq)\mathcal{M}_{\sigma}$, which follows by \cite[Prop. 4.6 $(1)$, Thm. 5.14 $(2)$, Lemma 5.16 $(1)$, Cor. 5.17 $(1)$]{compositionpaper}.

Second, by \cite[Proposition 5]{nuclearglobal2} we get that $$\mathcal{S}_{(\mathcal{M}_{\omega})}(\RR^d)=\mathcal{S}_{(\omega)}(\RR^d)\subseteq\mathcal{S}_{(\sigma)}(\RR^d)=\mathcal{S}_{(\mathcal{M}_{\sigma})}(\RR^d)$$ holds true for all dimensions $d\in\NN$ and the equalities hold as topological vector spaces.

In order to apply Theorem \ref{Beurlingprop} to the matrices $\mathcal{M}_{\omega}$ and $\mathcal{M}_{\sigma}$ we remark that by \cite[Appendix A]{sectorialextensions1},
we know that $\omega$ has \eqref{om7} if and only if $\mathcal{M}_{\omega}$ satisfies \eqref{Beurlingsquare} and $\omega$ is automatically non-quasianalytic, see \cite[Lemma A.1]{sectorialextensions1}.
\end{proof}

Let us finally remark that
in Theorems~\ref{Gelfand-inclusion2new}, \ref{Gelfand-inclusion2newmixed},
\ref{Gelfand-inclusionbeurlingmatrix} and \ref{Beurlingprop} we could avoid the assumption
that $\mathcal{N}$ is standard log-convex by substituting \eqref{Prop32Komatsu} with
\begin{equation}
N_p\geq\sup_{t\ge 0}\frac{t^p}{\exp(\omega_{\mathbf{N}}(t))},\;\;\;p\in\NN_0,
\end{equation}
which is always true (cf. \cite{BJOS-cvxminorant}).


\appendix

\section{On the weight $\omega(t)=\log(1+t)$ and the Schwartz class $\mathcal{S}$}\label{logweight}
The weight $\omega(t)=\log(1+t)$, $t\ge 0$, clearly satisfies the standard assumptions, $(\alpha)$ and $(\delta)$ in Definition \ref{defomega}, and $\omega(t)=o(t^2)$ which implies $(\beta)$. But the crucial property $(\gamma)$ (i.e. $\log(t)=o(\omega(t))$) fails and also \eqref{omega6} does not hold true.\vspace{6pt}

Note that $\omega$ is equivalent to $\omega_1(t):=\max\{0,\log(t)\}$ and the relation $\sim$ preserves all the properties listed in Definition \ref{defomega} except the convexity condition $(\delta)$ which is clear for $\omega_1$ as well. In addition $\omega_1$ is normalized, more precisely $\omega_1(t)=0$ for $0\le t\le 1$ and $\omega_1(t)=\log(t)$ for all $t\ge 1$. It is known that the weight $\omega$ yields the classical Schwartz class $\mathcal{S}$ and the aim of this appendix is to study for this limiting case the associated weight matrix.

More precisely we show that the weight matrix approach is not ``well-defined'' for the weights $a\log(1+t)$, $a>0$, because the matrix associated with such weights does not contain sequences of positive real numbers (as usually required).\vspace{6pt}

\begin{lemma}\label{appendixlemma1}
We get as topological vector spaces the identity $\mathcal{S}_{(\omega)}(\RR^d)=\mathcal{S}_{(\omega_1)}(\RR^d)=\mathcal{S}(\RR^d)$ for any dimension $d$. However, the analogous result for the Roumieu-type space fails and neither $\mathcal{M}_{\omega_1}$ nor $\mathcal{M}_{\omega}$ is a weight matrix as defined in Section \ref{weightfunctions}.
\end{lemma}

\begin{proof}
By normalization we obtain for the Young-conjugate for all $x\ge 0$ (because $\varphi_{\omega_1}(y)=\omega_1(e^y)=0$ for all $-\infty<y\le 0$):
\begin{align*}
\varphi^{*}_{\omega_1}(x)&=\sup_{y\in\RR}\{xy-\varphi_{\omega_1}(y)\}=\sup_{y\ge0}\{xy-\varphi_{\omega_1}(y)\}=\sup_{y\ge0}\{xy-y\}=\sup_{y\ge0}\{y(x-1)\}.
\end{align*}
Consequently, $\varphi^{*}_{\omega_1}(x)=0$ for all $0\le x\le 1$ and $\varphi^{*}_{\omega_1}(x)=+\infty$ for all $x>1$.

This implies for the matrix $\mathcal{M}_{\omega_1}:=\{\mathbf{W}^{(\lambda)}: \lambda>0\}$ that
$$\mathbf{W}^{(\lambda)}_{\alpha}:=\exp\left(\frac{1}{\lambda}\varphi^{*}_{\omega_1}(\lambda|\alpha|)\right)=1,\;\;\;|\alpha|\le\frac{1}{\lambda};\hspace{20pt}\mathbf{W}^{(\lambda)}_{\alpha}=+\infty,\;\;\:|\alpha|>\frac{1}{\lambda}.$$
Hence, for any given index $\lambda>0$ only for finitely many multi-indices $\alpha\in\NN^d_0$ we have that $W^{(\lambda)}_{\alpha}$ is not equal to $+\infty$. This ``length of finiteness'' is increasing when $\lambda\rightarrow 0$, and these small indices are important for the Beurling case. Note that for all $\lambda>1$ we have the situation that $W^{(\lambda)}_{\alpha}=+\infty$ for all $\alpha\in\NN^d_0$ with $|\alpha|\ge 1$. Thus only for $\alpha=0$ we get $W^{(\lambda)}_{\alpha}=1$.\vspace{6pt}

By using the matrix $\mathcal{M}_{\omega_1}$ we can describe the Beurling-type class $\mathcal{S}_{(\omega)}(\RR^d)=\mathcal{S}(\RR^d)$ as a topological vector space, however the Roumieu-type class is not fitting in this framework anymore.\vspace{6pt}

Note that by the equivalence between $\omega$ and $\omega_1$ it is straight-forward to see (e.g. see \cite[Lemma 5.16 $(1)$]{compositionpaper} that, when denoting $\mathcal{M}_{\omega}:=\{\mathbf{V}^{(\lambda)}: \lambda>0\}$, we get:
$$\exists\;H\ge 1\;\forall\;\lambda>0\;\exists\;C\ge 1\;\forall\;\alpha\in\NN^d_0:\;\;\;W^{(\lambda)}_{\alpha}\le C V^{(H\lambda)}_{\alpha}.$$
Hence also for these sequences we obtain the property that only finitely many multi-indices yield a finite value. Note that instead of $\omega$ we could consider any other (equivalent) weight $a\log(1+t)$, $a>0$.
\end{proof}

The next result completes this situation by starting with a sequence. For the proof of this result we refer to \cite[Lemma 7.2]{solidassociatedweight}.

\begin{lemma}\label{appendixlemma2}
Concerning the function $t\mapsto\log(1+t)$ we get:
\begin{itemize}
\item[$(i)$] There does not exist a sequence $\mathbf{M}\in{\mathcal{LC}}$ such that
\begin{equation}\label{logequi}
\omega_{\mathbf{M}}\sim t\mapsto\log(1+t).
\end{equation}
\item[$(ii)$] Let $\mathbf{M}=(M_p)_p$ be a sequence with $1=M_0$ and
$$\exists\;q_0\in\NN_{>0}\;\forall\;p>q_0:\;\;\;M_p=+\infty,$$
and such that $1\le\mu_p\le\mu_{p+1}$ for $1\le p\le q_0$. (So $M_p\in\RR_{>0}$ for only finitely many indices $p$ and $\mu_{q_0+1}=\frac{M_{q_0+1}}{M_{q_0}}=+\infty$). Then $\omega_{\mathbf{M}}$ satisfies \eqref{logequi} by using the conventions $\frac{1}{+\infty}=0$ and $\log(0)=-\infty$.
\end{itemize}
\end{lemma}

The first part means that any weight $a\log(1+t)$, $a>0$, cannot occur in the equivalence class of associated weight functions coming from (standard) weight sequences. However, the second part yields that for ``exotic sequences $\mathbf{M}$'' such that $M_p=+\infty$ for all but finitely many $p$, each weight $a\log(1+t)$, $a>0$, is equivalent to $\omega_{\mathbf{M}}$.

In particular, we can apply $(ii)$ in Lemma \ref{appendixlemma2} to $\mathbf{W}^{(\lambda)}\in\mathcal{M}_{\omega_1}$. Combining this with Lemma \ref{appendixlemma1} we have that also in this exotic setting \cite[Lemma 12]{nuclearglobal2} remains valid, so $\omega_{\mathbf{W}^{(\lambda)}}\sim\omega_1\sim\omega$ for all indices $\lambda>0$.

\vspace*{10mm}
{\bf Acknowledgments.}
Boiti and Oliaro have been partially supported by the INdAM - GNAMPA Projects 2020 ``Analisi microlocale e applicazioni: PDE's stocastiche e di evoluzione, analisi tempo-frequenza, variet\`{a}'' and 2023 ``Analisi di Fourier e Analisi Tempo-Frequenza
di Spazi Fun\-zio\-na\-li ed Ope\-ra\-to\-ri". Boiti is partially supported by the Projects FAR 2021, FAR 2022,
FIRD 2022 and FAR 2023 (University of Ferrara).
Jornet are partially supported by the project PID2020-\-119457GB-\-100 funded by MCIN/AEI/10.13039/501100011033 and by ``ERDF A way of making Europe''.  Schindl is supported by FWF (Austrian Science fund) project 10.55776/P33417.

\end{document}